\documentclass[a4paper]{amsart}

\usepackage{standalone}

\usepackage[all,cmtip]{xy}

\usepackage{diagbox}

\usepackage{tikz}
\usepackage{tikz-cd}
\usetikzlibrary{decorations.markings}
\tikzset{middlearrow/.style n args={5}{
		decoration={
			markings,
			mark=at position #1 with {\arrow[rotate=#2]{#3},\node[transform shape,#5] {#4};}},postaction={decorate}},
	middlearrow/.default={.5}{0}{>}{}{below}	}
\usetikzlibrary{calc}

\usepackage{amssymb}
\usepackage{mathtools}
\usepackage{thmtools,thm-restate}

\usepackage{etoolbox}

\theoremstyle{plain}
\newtheorem{thm}{Theorem}[section]
\newtheorem{prop}[thm]{Proposition}
\newtheorem{cor}[thm]{Corollary}

\newtheorem{lem}[thm]{Lemma}

\newtheorem{thmx}{Theorem}

\newtheorem{propx}[thmx]{Proposition}

\theoremstyle{definition}
\newtheorem{defn}[thm]{Definition}
\newtheorem{ex}[thm]{Example}
\newtheorem{rem}[thm]{Remark}

\usepackage{graphicx}
\graphicspath{{Images/}}
\usepackage{subcaption}

\usepackage[shortlabels]{enumitem}

\usepackage[british]{isodate}

\usepackage{bm}

\usepackage{hyperref}

\newcommand{\Z}{\mathbb{Z}}

\newcommand{\Q}{\mathbb{Q}}
\newcommand{\Ss}{\mathbb{S}}

\DeclareMathOperator{\Aut}{Aut}
\DeclareMathOperator{\Out}{Out}
\DeclareMathOperator{\Inn}{Inn}

\DeclareMathOperator{\Fix}{Fix}
\DeclareMathOperator{\SL}{SL}
\DeclareMathOperator{\GL}{GL}
\DeclareMathOperator{\Sp}{Sp}

\DeclareMathOperator{\Hom}{Hom}
\DeclareMathOperator{\MCG}{MCG}

\DeclareMathOperator{\CV}{CV}
\DeclareMathOperator{\Id}{Id}
\DeclareMathOperator{\con}{con}
\DeclareMathOperator{\ch}{ch}
\DeclareMathOperator{\Cl}{Cl}
\DeclareMathOperator{\Top}{Top}
\DeclareMathOperator{\SP}{Sp}
\DeclareMathOperator{\Mod}{Mod}
\DeclareMathOperator{\Alg}{Alg}
\DeclareMathOperator{\Res}{Res}
\DeclareMathOperator{\Map}{Map}

\DeclareMathOperator{\ev}{ev}
\DeclareMathOperator{\odd}{odd}
\DeclareMathOperator{\pri}{prime}

\title{On the Farrell--Tate $K$-theory of $\Out(F_n)$}

\author{Naomi Andrew}
\address[Naomi Andrew]{Mathematical Institute, Andrew Wiles Building, Observatory Quarter, University of Oxford, Oxford OX2 6GG, United Kingdom}
\email{Naomi.Andrew@maths.ox.ac.uk}

\author{Irakli Patchkoria}
\address[Irakli Patchkoria]{Department of Mathematics, University of Aberdeen, Fraser Noble Building, Aberdeen AB24 3UE, United Kingdom}
\email{irakli.patchkoria@abdn.ac.uk}

\date{}

\begin{document}

\begin{abstract} Using L\"uck's Chern character isomorphism we obtain a general formula in terms of centralisers for the $p$-adic Farrell--Tate $K$-theory of any discrete group $G$ with a finite classifying space for proper actions. We apply this formula to $\Out(F_n)$. The case $n=p+1$ turns out to be especially interesting for the following reason: Up to conjugacy there is exactly one order $p$ element in $\Out(F_{p+1})$ which does not lift to an order $p$ element in $\Aut(F_{p+1})$. We compute the rational cohomology of the centraliser of this element and as a consequence obtain a full calculation of the $p$-adic Farrell--Tate $K$-theory of $\Out(F_{p+1})$ for any prime $p \geq 5$. Our arguments provide an infinite family of $\Q_p$ summands in $K^1(B\Out(F_n)) \otimes_\Z \Q$, with no need for computer calculations: the first such summand is in $K^1(B\Out(F_{12})) \otimes_{\Z} \Q$. 

\end{abstract}

\maketitle

\section{Introduction}

The rational homology of $\Out(F_n)$ has been extensively studied in the last thirty years. The groups $H_i(\Out(F_n); \Q)$ have been computed for $n \leq 7$ and $i \in \Z$ by Vogtmann \cite{Vogtmann06}, Gerlits \cite{Gerlits}, Ohashi \cite{Ohashi}, and Bartholdi \cite{Bartholdi}. The main tool for these calculations is the cellular structure of the spine of the Culler--Vogtmann Outer space \cite{CullerVogtmann1986CVn} and its filtration due to Kontsevich \cite{Kon1, Kon2, ConantVogtmann}. 

Culler and Vogtmann \cite{CullerVogtmann1986CVn} showed that the virtual cohomological dimension of $\Out(F_n)$ is equal to $2n-3$. Hatcher and Vogtmann proved that $H_i(\Out(F_n); \Q)$ stabilises when $n$ is large enough \cite{HatcherVogtmann04} and Galatius showed that the stable value vanishes \cite{Galatius}. Further computations have been done towards understanding the classical rational and orbifold Euler characteristics of $\Out(F_n)$. The results of \cite{MTS, BorinskyVogtmann, BorinskyVogtmann1} imply that in general the rational homology $H_*(\Out(F_n); \Q)$ contains many odd dimensional classes (in fact more than even dimensional classes). However, writing down these odd classes explicitly seems to be very difficult. For $n \leq 6$ the rational homology $H_*(\Out(F_n), \Q)$ is even (see e.g. \cite{Ohashi}). A non-trivial odd dimensional group is computed in \cite{Bartholdi}, where Bartholdi shows that $H_{11}(\Out(F_7); \Q)$ is isomorphic to $\Q$. Bartholdi's result uses Kontsevich's filtration and computer calculations. No other odd dimensional non-trivial rational homology groups of $\Out(F_n)$ have been detected so far. 

In this paper we give a systematic way of producing odd dimensional classes in the rationalised complex topological $K$-theory of $B\Out(F_n)$. We show that for any prime $p \geq 11$, the $\Q$-vector space $K^1(B\Out(F_{p+1})) \otimes_{\Z} \Q$ is non-trivial by providing a quadratic lower bound $\frac{1}{24}(p-7)(p-5)$ for the $\Q_p$-dimension of its $p$-adic summand. We do not need any computer calculations to verify this. The main tools for our computations are the equivariant Chern character of L\"uck \cite{Luck2007} and the generalised Farrell--Tate cohomology defined by Klein in \cite{Klein2001, Klein2002}. Before we explicitly formulate our results, we review general results on rational $K$-theory and Farrell--Tate cohomology. 

L\"uck \cite{Luck2007} provided a conceptual way to compute the rational $K$-theory of $BG$, where $G$ is an infinite discrete group which admits a finite classifying $G$-space for proper actions. This classifying $G$-space is denoted by $\underline{E}G$. It is uniquely determined up to $G$-equivariant homotopy equivalence by the properties that it is a $G$-CW complex and the fixed subspace $\underline{E}G^H$ is contractible if $H$ is finite and empty otherwise. We say that $G$ \emph{admits a finite $\underline{E}G$} if there is a finite $G$-CW complex model for $\underline{E}G$. Equivalently, we can ask that the quotient $\underline{E}G/G$ is compact. For more details on $\underline{E}G$ and examples see \cite{LuckSurvey}. The following theorem of L\"uck is the main calculational tool in this paper \cite[Theorem 0.1]{Luck2007}: 

\begin{thm}[L\"uck] \label{thm:Wolfgang} Let $G$ be a discrete group with a finite $\underline{E}G$. Then for any $m\in \Z$, the equivariant Chern character map
\[\ch^m \colon K^m(BG) \otimes_{\Z} \Q \xrightarrow{\cong} \prod_{i \in \Z} H^{2i+m}(G; \Q) \times \prod_{p \; \pri } \prod_{[g] \in \con_{p}(G)} \prod_{i \in \Z} H^{2i+m}(C   \langle g \rangle; \Q_p) \]
is an isomorphism, where $\con_p(G)$ denotes the set of conjugacy classes of the non-trivial $p$ power order elements, and $C  \langle g \rangle$ denotes the centraliser of the cyclic subgroup generated by $g$. 
\end{thm}

\;

\noindent For virtually torsion-free groups with finite virtual cohomological dimension and certain geometric finiteness conditions, an analogous theorem was obtained by Adem \cite{Adem1992, Adem1993}.  

In this paper it will be convenient to mostly work with $p$-adic $K$-theory $K_p$ which sees the rational cohomology of $BG$ and the $p$-adic part of the above formula. The $p$-adic $K$-theory spectrum is defined as the $p$-completion of the complex $K$-theory spectrum \cite[Section 2]{Bousfield1979}. It follows from L\"uck's formula (see also \cite[Theorem 6.3]{Adem1993}) that the Chern character induces an isomorphism for any $m\in \Z$: 
\[\ch_p^m \colon K_p^m(BG) \otimes_{\Z} \Q \xrightarrow{\cong} \prod_{i \in \Z} H^{2i+m}(G; \Q_p) \times \prod_{[g] \in \con_{p}(G)} \prod_{i \in \Z} H^{2i+m}(C   \langle g \rangle; \Q_p). \]

\noindent Various calculations have been done using this formula. Groups to which this formula was applied include certain graph products, arithmetic groups and mapping class groups (see e.g. \cite{Adem1992, Adem1993, LPS2024}). We list some of them: the amalgamated product $L \ast_H M$, where $H \leq L,M$, and $H,L,M$ are finite, $\SL_3(\Z)$, $\SL_2(\mathcal{O}_F)$, where $F$ is a totally real field, $\SL_{p-1}(\Z)$, $\SP_{p-1}(\Z)$, for $p$ an odd prime, the mapping class group $\MCG(\Sigma_{\frac{p-1}{2}})$, for $p \geq 5$, and the right angled Coxeter group $W(L)$, where $L$ is a finite graph. (An alternative approach using a concrete cellular structure was used in \cite{DegrijseLeary} to compute the complex $K$-theory of $BW(L)$.)

The Chern character isomorphism of L\"uck shows that it is more difficult to compute the rational ($p$-adic) $K$-theory of $BG$ than the rational cohomology of $BG$ since the latter is a direct factor of the former. However, one could expect that if the centralisers are not too large, then the factor 
\[\prod_{[g] \in \con_{p}(G)} \prod_{i \in \Z} H^{2i+m}(C   \langle g \rangle; \Q_p)\]
might be easier to compute. This is indeed the case for certain groups and primes as we will see below. But before we mention these computations, we give a conceptual interpretation of the latter summand in terms of Klein's Farrell--Tate $K$-theory \cite{Klein2001, Klein2002}. 

Let $G$ be a virtual duality group in the sense of \cite[Section VIII.11]{BrownCohomology} with finite virtual cohomological dimension $d$. The classical Farrell--Tate cohomology groups $\widehat{H}^*(G,M)$, defined in \cite{Farrell} and \cite{BrownCohomology}, measure the failure of the duality map
\[H_*(G,D_G \otimes_{\Z} M) \to H^*(G,M)\]
being an equivalence. Here $D_G$ denotes the dualising module $H^d(G; \Z[G])$ and $M$ is any $\Z[G]$-module. If the Farrell--Tate cohomology vanishes for a general $M$, then $G$ is a Bieri--Eckmann duality group \cite{BieriEckmann}. If $M$ is a $\Q[G]$-module, then the Farrell--Tate cohomology groups $\widehat{H}^*(G,M)$ vanish, witnessing the fact that virtual duality groups are rational duality groups \cite[Theorem 9.2]{Bieri} (see also \cite[Section X.3]{BrownCohomology}). 

Now if we take any generalised cohomology theory $E^*(-)$ and a group $G$, following Klein \cite{Klein2001, Klein2002}, one can define the associated \emph{generalised Farrell--Tate cohomology} $\widehat{E}^*(BG)$ with coefficients in $E$. These groups measure the failure of duality in $E$-(co)homology. In the special case $E=K_p$, one obtains the \emph{$p$-adic Farrell--Tate $K$-theory groups} $\widehat{K_p}^*(BG)$. Before giving the technical definition of $\widehat{K_p}^*(BG)$, we observe that the importance of these groups becomes clear from the following proposition which we prove as Proposition \ref{prop:Tate K-theory} in Section \ref{Sec: Tate duality}:

\begin{propx}
    Let $G$ be a discrete group with a finite $\underline{E}G$ and $p$ a prime. Then for any $m\in \Z$, the equivariant Chern character induces an isomorphism
\[\widehat{K_p}^m(BG) \cong \prod_{[g] \in \con_{p}(G)} \prod_{i \in \Z} H^{2i+m}(C   \langle g \rangle; \Q_p).    \]
\end{propx}

The precise definition of the generalised Farrell--Tate cohomology is recalled in Section \ref{Sec: Klein}. Here we only give a sketch and compare it to the classical picture: In \cite{Klein2001} Klein introduces a spectrum $D_G$ with a $G$-action, called the \emph{dualising spectrum}. For any spectrum $E$ with a $G$-action, he also constructs the \emph{norm map}
\[ N_G \colon D_G \otimes_{hG} E \to E^{hG},\]
going from the homotopy orbits to the homotopy fixed points \cite{Klein2001}. The cofibre of this map is denoted by $E^{tG}$ and is called the \emph{generalised Farrell--Tate spectrum}. These objects and constructions are direct higher homotopical generalisations of classical objects and constructions in group cohomology. The dualising spectrum $D_G$ is a generalisation of the dualising module $H^d(G; \Z[G])$ of a virtual duality group $G$ with virtual cohomological dimension $d$. Homotopy orbits and fixed points are generalisations of (derived) $G$-coinvariants and $G$-invariants, respectively. The homotopy groups of the Farrell--Tate spectrum generalise the classical Farrell--Tate cohomology. The $p$-adic Farrell--Tate $K$-theory groups $\widehat{K_p}^*(BG)$ of $G$ are then defined to be the homotopy groups $\pi_{-\ast} (K_p^{tG})$, where $G$ acts trivially on $K_p$. In Section \ref{Sec: Klein} we show that if $G$ has a finite $\underline{E}G$, then the groups $\widehat{K_p}^*(BG)$ are rational (that is, they are vector spaces over $\Q$). 

\quad

As observed above, the Farrell--Tate $K$-theory measures the failure of duality in $K$-theory. Proposition A implies that $\widehat{K_p}^0(BG)$ never vanishes unless $G$ is $p$ torsion-free. However, $\widehat{K_p}^1(BG)$ can vanish and this leads to a notion of \emph{$p$-adic weak duality} which we introduce in Section \ref{Sec: Tate duality}. We show that $\widehat{K_p}^1(BG)$ vanishes if and only if the norm map
\[N_G \otimes_{\Z} \Q \colon (K_p)_{-1}(BG; D_G) \otimes_{\Z} \Q  \to K_p^1(BG) \otimes_{\Z} \Q \]
is an isomorphism. In Section \ref{Sec: Tate duality} we give examples of groups which satisfy this weak form of duality: the arithmetic groups $\SL_3(\Z)$, $\SP_{p-1}(\Z)$, and the mapping class group $\MCG(\Sigma_{\frac{p-1}{2}})$ (where $p \geq 5$). On the other hand we show that $\GL_{p-1}(\Z)$ for $p \geq 5$ does not satisfy this weak duality. We also show that the difference between the $\Q_p$-dimensions of $\widehat{K_p}^0(BG)$ and $\widehat{K_p}^1(BG)$ can be arbitrarily large. 

After understanding the generalities of the $p$-adic Farrell--Tate $K$-theory, we apply Proposition A to $\Out(F_n)$. In Section \ref{Sec: Main calc} we make a full calculation of $\widehat{K_p}^*(B\Out(F_n))$ in the following cases: 
\begin{itemize}
    \item For $p \geq 3$ and $n=p-1,p$;
    \item For $p \geq 5$ and $n=p+1,p+2$;
    \item For $p \geq 7$ and $n=p+3$. 
\end{itemize}
In this range $\Out(F_n)$ does not contain subgroups isomoprhic to  $\Z/p \times \Z/p$ or $\Z/p^2$ and the classical Farrell--Tate cohomology is computed in \cite{GloverMislin2000, Chen97}. The Farrell--Tate $K$-theory has not yet been explored. Chen's thesis \cite{Chen97} classifies conjugacy classes of order $p$ elements in this range excluding $n=p+1$ and shows that these elements are given by isometries of the rose and theta type graphs. Chen also computes the centralisers of these elements. This allows us to almost immediately compute $\widehat{K_p}^*(B\Out(F_n))$ in the above range when $n \neq p+1$. In particular, $\widehat{K_p}^1(B\Out(F_n))=0$ unless $n = p+1$. For $n=p+3$, we also need a recent calculation of $H^*(\Out(F_3); \bigwedge^{2} H_1(F_3;\Q))$ due to Satoh \cite{Satoh24}. 

The case $n = p+1$ requires a non-trivial geometric input. Chen's thesis provides 3 conjugacy classes of order $p$ elements: one given by an isometry of a rose $R_p$, and the other two by isometries of theta graphs named $\theta_{02}$ and $\theta_{11}$. All these admit obvious order $p$ lifts to $\Aut(F_{p+1})$. We show in Section \ref{section: Centraliser of Phi} that there is exactly one more conjugacy class of order $p$ elements, represented by an element which we denote by $\Phi$. This class does not admit an order $p$ lift to $\Aut(F_{p+1})$ which makes it non-trivial to study: it is induced by a rotation of the graph of Figure~\ref{fig:rep_intro}.

\begin{figure}
    \centering
    \begin{tikzpicture}[line width=1pt]
	\foreach \x in {-60, 0,60,120,180,240}{
		\draw[middlearrow={.55}{-20}{>}{}{}] (\x:2) to[out=\x-30, in=\x+30,loop,distance=1.5cm] (\x:2);
		\node[inner sep=2,circle,draw,fill] at (\x:2){};
        }
	\foreach \x in {-60, 0,60,120,180}{
		\draw (\x:2) edge[middlearrow={0.55}{0}{>}{}{}] (\x+60:2);}
	\draw[dotted] (250:1.9) to (290:1.9); \node at (270:1.5) {$p$};
\end{tikzpicture}
    \caption{}
    \label{fig:rep_intro}
\end{figure}

In Section \ref{section: Centraliser of Phi} we classify the dimensions where such elements can occur and compute the rational homology of the centraliser of $\Phi$. In order to do this, we define a map with finite kernel from $C\langle \Phi \rangle$ to a finite index subgroup of $\Out(F_2)$. This enables us to define an action of $C \langle\Phi\rangle$ on a tree (the reduced spine of $\CV_2$). The stabilisers in this action are finite, and so a Mayer--Vietoris argument shows that the rational homology agrees with that of the quotient graph. To calculate it, we must compute the number of edge and vertex orbits, which we do by means of an equivalent computation in finite groups. The rational homology is given in Corollary~\ref{cor: rational homology of CPhi}. We note that the element $\Phi$ has also been studied in \cite{BridsonPiwek}, where they show that it is the unique conjugacy class in $\Out(F_{p+1})$ defining a free-by-cyclic group isomorphic to $F_2 \times \Z$.

By combining the results of Section \ref{section: Centraliser of Phi} with Proposition A, we obtain the following theorem, which we prove as Theorem \ref{them: main computation} in Section \ref{Sec: Main calc}:

\begin{thmx}
    Let $p \geq 5$ be a prime. Then we have
\[\widehat{K_p}^0(B\Out(F_{p+1})) \cong \Q_p^4\]
and 
\[\widehat{K_p}^1(B\Out(F_{p+1})) \cong \begin{cases} 0, & \text{if}  \;\; p=5,7\\ \Q_p^{\frac{1}{24}(p-7)(p-5)}, &  \text{if} \;\; p \geq 11  \end{cases}. \]
\end{thmx}

As a consequence, we see that when $p \geq 11$, then 
\[\dim_{\Q_{p}} K^1_p(B\Out(F_{p+1})) \otimes_{\Z} \Q \geq \dim_{\Q_p} \widehat{K_p}^1(B\Out(F_{p+1})) = \frac{1}{24}(p-7)(p-5).\]
In particular, this produces a family of non-trivial odd dimensional classes in the $K$-theory of $B\Out(F_n)$. The first class occurs for $p=11$, where we get
\[\dim_{\Q_{11}} K^1_{11}(B\Out(F_{12})) \otimes_{\Z} \Q \geq \dim_{\Q_{11}} \widehat{K_{11}}^1(B\Out(F_{12})) = 1 .\]

In the final section of the paper, we combine our calculations with the computations of Vogtmann, Hatcher, Brady, Gerlits, Ohashi, Bartholdi and Satoh \cite{Brady93, HatcherVogtmann98, Vogtmann02, Gerlits, Ohashi, Bartholdi, Satoh24}. We produce a table of low dimensional values of $\widehat{K_p}^*(B\Out(F_{n}))$ (see Table~\ref{t:Farrell--Tate}) and  $K_p^*(B\Out(F_{n})) \otimes_{\Z} \Q$ (see Table~\ref{t:p-adic}) which can be computed using currently available methods. For $p$-adic Farrell--Tate $K$-theory the first odd dimensional class appears in $\widehat{K_{7}}^1(B\Out(F_{11}))$. The existence of this class requires Gerlits' computer calculation of $H_7(\Aut(F_5); \Q)$, which is isomorphic to $\Q$. The next non-trivial odd class shows up in $\widehat{K_{11}}^1(B\Out(F_{12}))$ which is the class we construct explicitly using the Chern character. Further, Bartholdi's computer calculation \cite{Bartholdi} implies that $K^1(B\Out(F_{7})) \otimes_{\Z} \Q$ is non-trivial and the above argument shows that $K^1(B\Out(F_{11})) \otimes_{\Z} \Q$ is non-zero. However, again the first non-trivial odd class in rational $K$-theory which we can see without computer calculations lives in $K^1(B\Out(F_{12})) \otimes_{\Z} \Q$.

\subsection*{Acknowledgements}
We would like to thank Takao Satoh and Ric Wade for helpful conversations. We would also like to thank the University of Oxford, the University of Aberdeen and the Isaac Newton Institute for Mathematical Sciences for hospitality. 

This work has received funding from the European Research Council (ERC) under the European Union's Horizon 2020 research and innovation programme (Grant agreement No. 850930). The first named author was supported by the Royal Society of Great Britain, and the second named author was supported by the EPSRC grant EP/X038424/1 ``Classifying spaces, proper actions and stable homotopy theory''. This work was also supported by the EPSRC grant no EP/R014604/1.  

\section{Preliminaries from homotopy theory} We will use the language of model categories for simplifying the exposition. See for example \cite{Quillen, Hovey} for the background in model categories. All the homotopy theoretic constructions in this paper can be also done with $\infty$-categories. 

The category of topological spaces is denoted by $\Top$. This is a model category with weak homotopy equivalences as weak equivalences  and Serre fibrations as fibrations. Typical cofibrant objects are cellular complexes. In particular, any CW-complex is cofibrant. 

Let $G$ be a discrete group. We let $\Top_G$ denote the category of $G$-spaces. This can be equipped with a projective model structure \cite{SchwedeShipley}, where weak equivalences are $G$-equivariant maps which are underlying non-equivariant weak equivalences. Cofibrant objects are given by cell complexes with a free $G$-action. In particular, if $X$ is a CW-complex with a $G$-action, then the projection map $X \times EG \to X$ is a cofibrant replacement. 

We will also need the category of spectra $\Sp$. Here we can take any symmetric monoidal model category of spectra such as symmetric or orthogonal spectra equipped with the 
projective model structure, see e.g. \cite{HSS, MMSS}. The symmetric monoidal structure on $\Sp$ is given by the smash product which we denote by $\otimes$. The unit of this structure is the sphere spectrum $\Ss$. Tautologically, the category $\Sp$ is equivalent to $\Ss-\Mod$. For non-experts, one can think of the symmetric monoidal category of spectra as the higher homotopical algebra analog of the category of abelian groups with the tensor product. The unit $\Ss$ then behaves like $\Z$, the unit of the tensor product. Given a spectrum $E$ and an integer $n$, one can define the $n$-th \emph{(stable) homotopy group} $\pi_nE$ as the group of homotopy classes of maps from the $n$-th suspension of $\Ss$ to $E$. In particular, one has $\pi_0\Ss \cong \Z$. A spectrum $E$ is called \emph{rational} if $\pi_nE$ is a rational vector space for any $n \in \Z$. 

A \emph{ring spectrum} (or an $\Ss$-\emph{algebra}) is an associative monoid in $\Sp$. The category $\Ss-\Alg$ of ring spectra can be equipped with the projective model structure where fibrations and weak equivalences are those morphisms in $\Ss-\Alg$ which are underlying fibrations and weak equivalences in $\Sp$, respectively. 
By \cite[Theorem 4.1(3)]{SchwedeShipley}, a cofibration of ring spectra is a cofibration of the underlying spectra if the source is cofibrant as a spectrum. 
For any ring spectrum $R$, the category of left and right modules $R-\Mod$ and $\Mod-R$ can be also equipped with projective model structures where fibrations and weak equivalence are given by the underlying notions on spectra \cite{SchwedeShipley}.

The categories of spaces and spectra are related by an adjunction. The suspension spectrum functor
\[\Sigma^{\infty}_+ \colon \Top \to \Sp  \]
is a left adjoint functor which is a left Quillen functor (here $+$ stands for adding a disjoint basepoint). In particular, it preserves cofibrations and weak equivalences between CW-complexes. 

Given a discrete group $G$, the suspension spectrum $\Sigma^{\infty}_+G$ of the discrete space $G$ becomes a ring spectrum using the multiplication of $G$. We denote this ring spectrum by $\Ss[G]$. One can think of $\Ss[G]$ as the higher algebra analog of the group ring $\Z[G]$. In particular, we have $\pi_0\Ss[G] \cong \Z[G]$. 

The category of abelian groups with a left (right) $G$-action is equivalent to the category of left (right) modules over $\Z[G]$. Analogously, the category of $G$-objects in $\Sp$ is equivalent to the module category $\Ss[G]-\Mod$. This is equipped with the standard model structure from \cite{SchwedeShipley} as discussed above. The suspension spectrum extends to a left Quillen functor on the level of $G$-objects
\[\Sigma^{\infty}_+ \colon \Top_G \to \Ss[G]-\Mod.\]
In particular, if $X$ is a CW-complex with a $G$-action, then $\Sigma^{\infty}_+(EG \times X)$ is a cofibrant replacement of $\Sigma^{\infty}_+X$. An object in $\Ss[G]-\Mod$ will be called a \emph{$G$-spectrum} for short. 

Let $H \leq G$ be a subgroup. The the \emph{restriction functor} \[\Res^G_H \colon  \Ss[G]-\Mod \to \Ss[H]-\Mod\]
is a right Quillen functor whose left adjoint is given by the \emph{induction functor}
\[G \otimes_{H}- \colon \Ss[H]-\Mod  \to \Ss[G]-\Mod .\]
These functors are analogs of the classical restriction and induction functors between $\Z[G]$-modules and $\Z[H]$-modules. 

We will also need the mapping spectrum functor which is the analog of the $\Hom$ functor in abelian groups. The smash product $\otimes$ in $\Sp$ has a right adjoint, the \emph{internal mapping spectrum functor} $\Map(-,-)$. This functor can be used to define the generalised cohomology theory associated to any spectrum. Given a (fibrant or equivalently omega) spectrum $E$ and a CW-complex $X$, the \emph{$n$-th $E$-cohomology group} of $X$, denoted by $E^n(X)$, is defined to be $\pi_{-n}\Map(\Sigma^{\infty}_+ X, E)$. 

If $E$ and $E'$ are $G$-spectra, then $\Map_G(E,E')$ denotes the mapping spectrum of $G$-equivariant maps. This is the fixed point spectrum of the mapping spectrum $\Map(E,E') \in \Ss[G]-\Mod$ where the action is given via the conjugation. 

Finally, we will crucially use the orbit and fixed point functors. The inflation functor
\[\Sp \to \Ss[G]-\Mod \]
which equips any spectrum with a trivial $G$-action has both left and right adjoints. The right adjoint sends a $G$-spectrum $E$ to the fixed point spectrum $E^G$ and the left adjoint to the orbit spectrum $E_G$. These functors are not always homotopically well-behaved, though: they are not homotopy invariant on underlying fibrant and cofibrant objects respectively. Because of this one introduces homotopy invariant versions: the \emph{homotopy fixed points} $E^{hG}$ of $E$ is defined to be the spectrum $\Map_G(\Sigma^{\infty}_{+}EG,E)$. This functor preserves weak equivalences between underlying fibrant objects. Similarly the \emph{homotopy orbits} $E_{hG}$ is defined to be the orbit spectrum $(E \otimes \Sigma_+^{\infty}EG)_G$. It sends weak equivalences between underlying cofibrant objects to weak equivalences. In particular, the action on $E$ does not have to be free for invariance of homotopy orbits. Again these definitions are motivated by the analogous notions in $\Z[G]$-modules, namely $G$-invariants and $G$-coinvariants. In what follows we will write $E^{hG}$ and $E_{hG}$ for not necessarily underlying fibrant and cofibrant $E$ implicitly replacing $E$ if necessary. 


\section{Klein's dualising complex and the norm map}\label{Sec: Klein}

In \cite{Klein2001} Klein introduced a $G$-spectrum $D_G$ called the \emph{dualising spectrum} and the \emph{norm map}
\[ N_G \colon D_G \otimes_{hG} E \to E^{hG},\]
where $E$ is any $G$-spectrum. In this section we recall Klein's construction and prove the following:

\begin{prop} \label{prop: norm equivalence}  Let $G$ be a discrete group with a finite $\underline{E}G$. Suppose $E$ is a ring spectrum which is underlying cofibrant and let $h \colon \Ss[G] \to E$ be a morphism of ring spectra. If for any finite subgroup $H \leq G$, the norm map
\[N_H \colon E_{hH} \to E^{hH}\]
is an equivalence, then the norm map
\[ N_G \colon D_G \otimes_{hG} E \to E^{hG},\]
is an equivalence. 
\end{prop}

We offer here a somewhat alternative description of the norm map in the special case when $E$ is a ring spectrum and the $G$-action on the underlying spectrum of $E$ comes from a ring spectrum map $h \colon \Ss[G] \to E$. This is motivated by the following simple algebraic observation: suppose $G$ is a finite group, $R$ is a ring and we are given a ring homomorphism $h \colon \Z[G] \to R$. We equip the underlying abelian group of $R$ with a left $G$-action via $h$. Then the group $R^G$ of $G$-invariants is a right $R$-module via the multiplication of $R$ from the right and hence the fixed point map
$h^G \colon \Z[G]^G \to R^G$ can be extended to a map
\[\Z[G]^G \otimes_{\Z} R \to R^G. \]
Composing with the isomorphism $\Z \cong \Z[G]^G$ sending $1$ to the norm element $\sum_{g \in G} g \in \Z[G]$, one recovers the norm map
$N \colon R \to R^G$ which then factors over the $G$-coinvariants $R_G$. 

Now we start with the setup of Proposition \ref{prop: norm equivalence}. We use the fibrant replacement functor $(-)^f$ on the category of ring spectra $\Ss-\Alg$, which comes with an acyclic cofibration $E \to E^f$. If $E$ is cofibrant as a spectrum, then by \cite[Theorem 4.1(3)]{SchwedeShipley}, the map $E \to E^f$ is a cofibration of spectra. This implies that $E^f$ has a cofibrant underlying spectrum too. 

\begin{defn}[\cite{Klein2001}] \label{def: dualising} Let $G$ be a discrete group. \emph{Klein's dualising spectrum} $D_G$ is the defined to be the homotopy fixed points spectrum
\[{(\Ss[G]^f)}^{hG}.\]
This spectrum is equipped with a right $G$-action via the right $G$-action on $\Ss[G]^f$.  
\end{defn}

In particular if $H$ is a finite group, by \cite[Theorem 10.1]{Klein2001}, the dualising spectrum is $H$-equivariantly equivalent to the sphere spectrum $\Ss$ with the trivial action and Klein's norm map 
\[ N_H \colon D_H \otimes_{hH} E \to E^{hH},\]
specialises to the norm map
\[N_H \colon E_{hH} \to E^{hH}.\]
By \cite[Corollary 10.2]{Klein2001}, this map is an equivalence for $H$-spectra of the form $\Ss[H] \otimes Y$, where $Y$ is any $H$-spectrum.

Let $X$ be a finite proper $G$-CW complex. The morphism of ring spectra
\[h^f \colon \Ss[G]^f \to E^f \]
induces a morphism on mapping spectra
\[h^f_* \colon \Map_G(\Sigma^{\infty}_+(EG \times X), \Ss[G]^f) \to \Map_G(\Sigma^{\infty}_+(EG \times X), E^f).\]
The target of this morphism is a right $E^f$-module using the multiplication of $E^f$ from the right. Hence $h_*^f$ induces a map
\[\overline{N} \colon \Map_G(\Sigma^{\infty}_+(EG \times X), \Ss[G]^f) \otimes E^f \otimes \Sigma^{\infty}_+EG  \to \Map_G(\Sigma^{\infty}_+(EG \times X), E^f),\]
where $EG$ gets projected down to the base-point. Analogously to the homotopy fixed points of $\Ss[G]$, the spectrum $\Map_G(\Sigma^{\infty}_+(EG \times X), \Ss[G]^f)$ has a right $G$-action induced by the right $G$-action on $\Ss[G]^f$. It is now straightforward to check that the map $\overline{N}$ factors through $G$-coinvariants yielding the norm map
\[N_G \colon \Map_G(\Sigma^{\infty}_+(EG \times X), \Ss[G]^f) \otimes_{hG} E^f  \to \Map_G(\Sigma^{\infty}_+(EG \times X), E^f).\]
It follows directly from the construction in \cite[Section 3]{Klein2001} that if we take $X=\underline{E}G$, then the map $N_G$ coincides with Klein's norm map. 

\qquad

\emph{Proof of Proposition \ref{prop: norm equivalence}.} We will show that the map 
\[N_G \colon \Map_G(\Sigma^{\infty}_+(EG \times X), \Ss[G]^f) \otimes_{hG} E^f  \to \Map_G(\Sigma^{\infty}_+(EG \times X), E^f)\]
is an equivalence for any finite proper $G$-CW complex $X$. Taking $X=\underline{E}G$ then will yield the statement of Proposition \ref{prop: norm equivalence}. The map $N_G$ is natural in $X$ and the source and target of $N_G$ send finite homotopy colimits to finite homotopy limits. Hence to show that $N_G$ is an equivalence, it suffices to show that it is an equivalence for $X=G/H$, where $H \leq G$ is any finite subgroup. By the restriction-induction adjunction the map
\[N_G \colon \Map_G(\Sigma^{\infty}_+(EG \times G/H), \Ss[G]^f) \otimes_{hG} E^f  \to \Map_G(\Sigma^{\infty}_+(EG \times G/H), E^f)\]
is equivalent to the map
\[ \Map_H(\Sigma_{+}^{\infty} EH, \Ss[G]^f) \otimes_{hG} E^f  \to \Map_H(\Sigma^{\infty}_+EH, E^f).\]
This map fits into a commutative diagram
\[\xymatrix{\Map_H(\Sigma_{+}^{\infty} EH, \Ss[G]^f) \otimes_{hG} E^f  \ar[r] & \Map_H(\Sigma^{\infty}_+EH, E^f) \\ \Map_H(\Sigma_{+}^{\infty} EH, \Ss[H]^f) \otimes_{hH} E^f  \ar[ur]_{N_H} \ar[u]^{\iota^H_*}, & & }\]
where the left vertical map is induced by the inclusion $\iota^H \colon \Ss[H] \to \Ss[G]$ and the diagonal map is Klein's norm for the finite subgroup $H$. Hence to show that $N_G$ is an equivalence it suffices to show that $\iota_*^H$ is an equivalence. By \cite[Corollary 10.2]{Klein2001} and naturality of the norm, $\iota_*^H$ is equivalent to the map 
\[\Ss[H \setminus H] \otimes_{hH} E^f \to \Ss[H \setminus G] \otimes_{hG} E^f \]
induced by the inclusion $H \leq G$. By definition of homotopy orbits this map is equivalent to the identity 
\[\Id \colon (E^f)_{hH} \to (E^f)_{hH}\]
which completes the proof. \qed

The following is due to Klein \cite{Klein2001}:

\begin{defn} The \emph{Farrell--Tate spectrum} of $G$ with coefficients in $E$ is the cofibre of the norm map and is denoted by $E^{tG}$. 
\end{defn}

The homotopy groups $\pi_{-\ast}(E^{tG})$ can be thought as the Farrell--Tate cohomology groups with coefficients in $E$. In what follows we denote $\pi_{-\ast}(E^{tG})$ by $\widehat{E}^*(BG)$.

By definition, we have a cofibre sequence 
\[\xymatrix{D_G \otimes_{hG} E \ar[r]^-{N_G} & E^{hG} \ar[r] & E^{tG}.}\]
Hence the norm is an equivalence if and only if the Farrell--Tate spectrum vanishes. 

\begin{ex} \label{rational vanishing} Suppose that we are given $G$, $E$, and $h \colon \Ss[G] \to E$ as in Proposition \ref{prop: norm equivalence}. Assume additionally that $E$ is rational. Then the Farrell--Tate spectrum $E^{tG}$ vanishes. Indeed this follows from Proposition \ref{prop: norm equivalence} and from the following two facts: Finite groups are rationally acyclic and the classical norm map from coinvariants to invariants is a rational isomorphism for finite groups.  
\end{ex}

Now let $K_p$ denote the $p$-adic complex $K$-theory spectrum which is a (cofibrant) ring spectrum. The augmentation $\Ss[G] \to \Ss$ composed with the unit $\Ss \to K_p$ gives a map of ring spectra $\Ss[G] \to K_p$ which equips $K_p$ with the trivial $G$-action. 

\begin{prop} \label{Tate K-theory is rational} Let $G$ be a discrete group with finite $\underline{E}G$. Then the Farrell--Tate spectrum $K_p^{tG}$ is rational. 
\end{prop}

\begin{proof}
Since $K_p^{tG}$ is $p$-local, it suffices to show that the mod $p$ reduction $(K_p/p)^{tG}$ vanishes. By the Adams splitting we have
\[K_p/p \simeq \bigvee_{i=0}^{p-2} \Sigma^{2i} K(1), \]
where $K(1)$ is the first Morava $K$-theory ring spectrum at the prime $p$.  Again the unit precomposed with the augmentation gives a morphism of ring spectra 
\[\Ss[G] \to K(1)\]
which equips $K(1)$ with the trivial $G$-action. Since the norm map is natural, it further suffices to show that $K(1)^{tG}$ vanishes. This now follows from Proposition \ref{prop: norm equivalence} and a result of Greenlees-Sadofsky \cite{GreenleesSadofsky} which says that $K(1)^{tH}$ vanishes for $H$ a finite group. \end{proof}

\begin{rem} Let $K(n)$ denote the $n$-th Morava $K$-theory at the prime $p$. A forthcoming paper joint with Gijs Heuts will show that for any discrete group $G$ with a finite $\underline{E}G$ and a $K(n)$-local spectrum $E$ with a $G$-action, the Farrell--Tate spectrum
$E^{tG}$ vanishes $K(n)$-locally. (See also \cite{CCRY} for a more general result on Tate vanishing.) We will also give an alternative construction of $D_G$ in terms of the $\infty$-category of proper $G$-spectra of \cite{DHLPS} and \cite{LNP} and give an $\infty$-categorical proof of a more general version of Proposition \ref{prop: norm equivalence} which does not require $E$ to be a ring spectrum. We do not need these generalisations in the present paper though.

\end{rem}

\section{Farrell--Tate \texorpdfstring{$K$}{K}-theory and duality in complex \texorpdfstring{$K$}{K}-theory} \label{Sec: Tate duality}

In this section we give an explicit formula for the $p$-adic Farrell--Tate $K$-theory using L\"uck's Chern character. We will also explain how exactly Farrell--Tate $K$-theory measures failure of duality in $K$-theory. 

\begin{prop} \label{prop:Tate K-theory} Let $G$ be a discrete group with a finite $\underline{E}G$ and $p$ a prime. Then for any $m \in \Z$, the equivariant Chern character induces an isomorphism
\[\widehat{K_p}^m(BG) \cong \prod_{[g] \in \con_{p}(G)} \prod_{i \in \Z} H^{2i+m}(C   \langle g \rangle; \Q_p).   \]
\end{prop}

\begin{proof} Let $K_{\Q_p}$ denote the rationalisation of the $p$-adic $K$-theory spectrum $K_p$. The Farrell--Tate spectrum $K_p^{tG}$ is rational by Proposition \ref{Tate K-theory is rational}. By naturality of the norm map \cite{Klein2001}, we have a commutative diagram of spectra with exact rows in the stable homotopy category:
\[\xymatrix{(D_G \otimes_{hG} K_p)  \otimes \Q \ar[d]_{\simeq} \ar[r]^-{N_G \otimes \Q} & K_p^{hG} \otimes \Q \ar[d] \ar[r] & K_p^{tG} \ar[d] \\ D_G \otimes_{hG} K_{\Q_p}   \ar[r]^-{N_G} & K_{\Q_p}^{hG} \ar[r] & K_{\Q_p}^{tG}.}\]
The left vertical map is an equivalence since the rationalisation commutes with (homotopy) colimits.
Hence the construction and naturality of the Chern character \cite{Luck2005} and Theorem \ref{thm:Wolfgang} imply that it suffices to show that the left lower map is an equivalence. 
This follows from Example \ref{rational vanishing}. \end{proof}

\begin{ex} If $G$ is a finite group, then all the centralisers are finite and hence rationally acyclic. We recover a well-known computation of the Tate construction for equivariant $K$-theory for finite groups (see e.g. \cite[page 125]{JackowskiOliver} and \cite[Theorem 13.1]{GreenleesMay}):
\[\widehat{K_p}^{2m}(BG) \cong \prod_{[g] \in \con_{p}(G)} \Q_p,\]
and
\[\widehat{K_p}^{2m+1}(BG) =0,\]
for any integer $m \in \Z$. 
\end{ex}

Klein has observed \cite{Klein2001}, that the Farrell--Tate spectrum $E^{tG}$ sees the failure of (Poincar\'e) duality in $E$-theory. Since topological $K$-theory is $2$-periodic, Proposition \ref{prop:Tate K-theory} allows us to measure exactly how significant this failure of duality is. In fact, in certain cases we have some weak forms of duality. 

For convenience we denote the homotopy groups $\pi_*(D_G \otimes_{hG} K_p)$ of the homotopy orbit spectrum by $(K_p)_*(BG; D_G)$. One can think of these groups as the $p$-adic $K$-homology groups with coefficients in the local coefficient system determined by $D_G$. Then by definition and Bott periodicity, we have a six-term exact sequence
\[\xymatrix{(K_p)_0(BG; D_G) \ar[r] & K_p^0(BG) \ar[r] & \widehat{K_p}^0(BG) \ar[d] \\ \widehat{K_p}^1(BG) \ar[u] & K_p^1(BG) \ar[l] &  (K_p)_{-1}(BG; D_G). \ar[l] }\]
In particular, we see from this long exact sequence that the rationalised norm map
\[(K_p)_{-*}(BG; D_G) \otimes_{\Z} \Q  \xrightarrow{N_G\otimes_{\Z} \Q} K_p^{*}(BG) \otimes_{\Z} \Q \]
is never an isomorphism of graded $\Q$-vector spaces unless the group $G$ is $p$-power torsion-free. Indeed by Proposition \ref{prop:Tate K-theory}, the group $\widehat{K_p}^0(BG)$ always contains the product $\prod_{[g] \in \con_{p}(G)} \Q_p$ as a retract. 

\begin{defn} Let $G$ be a discrete group with a finite $\underline{E}G$ and $p$ a prime. We say that $G$ satisfies \emph{weak duality in $p$-adic $K$-theory} if the map 
\[N_G \colon (K_p)_0(BG; D_G) \to K_p^0(BG)\]
is injective and the map 
\[N_G \colon (K_p)_{-1}(BG; D_G) \to K_p^1(BG)\]
is surjective. 
\end{defn}

\begin{prop} \label{prop:Weak duality} Let $G$ be a discrete group with a finite $\underline{E}G$ and $p$ a prime. Then the following are equivalent:

\rm (i) The group $G$ satisfies weak duality in $p$-adic $K$-theory;

\rm (ii) The $p$-adic Farrell--Tate $K$-theory group $\widehat{K_p}^1(BG)$ vanishes;

\rm (iii) The map
\[N_G \otimes_{\Z} \Q \colon (K_p)_{-1}(BG; D_G) \otimes_{\Z} \Q  \to K_p^1(BG) \otimes_{\Z} \Q \]
is an isomorphism. 

\rm (iv) For any $p$-power order element $g \in G$, the centraliser $C\langle g \rangle$ has even rational cohomology. 

\end{prop}

\begin{proof} The equivalence $(i) \Longleftrightarrow (ii)$ is an easy consequence of the six term exact sequence. The equivalence $(ii) \Longleftrightarrow (iii)$ follows from the construction of the Chern character \cite{Luck2007}, Proposition \ref{Tate K-theory is rational}, and Proposition \ref{prop:Tate K-theory} which imply that the map 
\[K_p^*(BG) \otimes_{\Z} \Q \to \widehat{K_p}^{*}(BG) \]
is split surjective. The equivalence $(iv) \Longleftrightarrow  (ii)$ follows from Proposition \ref{prop:Tate K-theory}. 
\end{proof}

\begin{rem} Proposition \ref{prop:Weak duality} motivates more careful study of the rational cohomology of the centralisers $C\langle g \rangle$, where $g$ is a $p$-power order element. In particular, detecting non-trivial odd dimensional cohomology classes for the centralisers would automatically imply that $\widehat{K_p}^1(BG)$ does not vanish and hence the group $G$ does not have weak duality in $p$-adic $K$-theory. We also note that any non-trivial odd dimensional class in $\widehat{K_p}^1(BG)$ gives a non-trivial odd dimensional class in the rational $K$-theory $K^*(BG) \otimes_{\Z} \Q$ which does not come from the rational cohomology.  Below we will detect these kinds of odd dimensional classes in $K^*(B\Out(F_{p+1})) \otimes_{\Z} \Q$ when $p \geq 11$. For example, we will construct a class in $K^1(B\Out(F_{12})) \otimes_{\Z} \Q$ which does not come from $H^*(\Out(F_{12}); \Q)$.

When $G$ satisfies weak duality in $p$-adic $K$-theory, i.e. $\widehat{K_p}^1(BG)=0$, then $\widehat{K_p}^0(BG)$ measures the exact failure of honest duality. The smaller this is, the closer to honest duality we get. By the $6$-term exact sequence, the cokernel of $N_G \colon (K_p)_0(BG; D_G) \to K_p^0(BG)$ injects into $\widehat{K_p}^0(BG)$ and $\widehat{K_p}^0(BG)$ surjects onto the kernel of $N_G \colon (K_p)_{-1}(BG; D_G) \to K_p^1(BG)$. The best possible case is when all centralisers of $p$-power order elements are rationally acyclic. In this case $\widehat{K_p}^0(BG) \cong \prod_{[g] \in \con_{p}(G)} \Q_p$. 

\end{rem}

\section{Examples} 

In this section we compute some examples and study the behavior of their $p$-adic Farrell--Tate $K$-theory. In general all possibilities can occur: all centralisers can be rationally acyclic, $\widehat{K_p}^*(BG)$ can be even while some centralisers are not rationally acyclic, $\widehat{K_p}^1(BG)$ might not be trivial and the \emph{Farrell--Tate $K$-theory Euler characteristic} \[\widehat{\chi_{K_p}}(BG)=\dim_{\Q_p}\widehat{K_p}^0(BG) -\dim_{\Q_p}\widehat{K_p}^1(BG)\] can be arbitrarily small or large. Simple examples can be provided by direct products $G \times \Z/p$, where $G$ is a $p$-torsion free group with a finite $\underline{E}G$. For instance we could take the following examples to cover some of these cases: $G=F_n$ and $p$ any prime, $G=\Sp_4(\Z)$ and $p \geq 7$, $G=\SL_4(\Z)$ and $p \geq 7$, and $G=\SL_2(\Z)$ and $p=5$. We do not go into details of these examples because they are straightforward. All these examples have central order $p$ elements and hence centralisers are very large (in fact the full group). Our goal below is to provide interesting examples where centralisers are quite small, showing the usefulness of Proposition \ref{prop:Tate K-theory}.

We start with arithmetic groups. It follows from \cite{Ji2007} that any arithmetic group admits a finite $\underline{E}G$. 

\subsection{The general linear group \texorpdfstring{$\SL_3(\Z)$}{SL3(Z)}} The group $\SL_3(\Z)$ has only $2$ and $3$-primary torsion. The $K$-theory of $B\SL_3(\Z)$ is computed in \cite{TezukaYagita92} and is shown to be even. Soul\'e in \cite{Soule78} provided a finite cellular model of $\underline{E}\SL_3(\Z)$ and showed that $\underline{E}\SL_3(\Z)/\SL_3(\Z)$ is contractible. Hence, $\SL_3(\Z)$ is in particular rationally acyclic. The papers \cite{Adem1992} and \cite{LPS2024} show that the centralisers of prime power order elements are rationally acyclic. This implies that for any prime $p$, we have
\[\widehat{K_p}^1(B\SL_3(\Z))=0\]
and hence $\SL_3(\Z)$ satisfies weak duality in $p$-adic $K$-theory. By \cite{TezukaYagita92} and \cite{Tahara71}, we know the classification of finite order elements: there are $2$ conjugacy classes of order $3$ elements, $2$ conjugacy classes of order $2$ elements and $2$ conjugacy classes of order $4$ elements. Hence 
\[\widehat{K_2}^0(B\SL_3(\Z)) \cong \Q_2^4,\]
and 
\[\widehat{K_3}^0(B\SL_3(\Z)) \cong \Q_3^2.\] 

\subsection{The general linear group \texorpdfstring{$\GL_{p-1}(\Z)$}{GLp-1(Z)} for a prime \texorpdfstring{$p \geq 5$}{p at least 5}} It follows from \cite[Section 2]{Ash89} and \cite{LevittNicolas98} that any non-trivial $p$-subgroup of $\GL_{p-1}(\Z)$ is isomorphic to $\Z/p$. Further the paper \cite{LatimerMacDuffee33} (see also \cite{SjerveYang97, Ash89}) implies that the conjugacy classes of order $p$ elements are in bijection with the ideal class group $\Cl(\mathbb{Q}(\zeta_p))$. It follows from \cite[Lemma 4]{Ash89} that $C\langle A \rangle $ for any $A$ of order $p$ is isomorphic to the group of units in $\Z[\zeta_p]$ which by the Dirichlet unit theorem is isomorphic to
\[\Z/p \times \Z/2 \times \Z^{\frac{p-3}{2}}.\]
Hence by Proposition \ref{prop:Tate K-theory}, we get
\[\widehat{K_p}^0(B\GL_{p-1}(\Z))=(\Q_p)^{\vert \Cl(\mathbb{Q}(\zeta_p) \vert \cdot 2^{\frac{p-5}{2}}} ,\]
and 
\[\widehat{K_p}^1(B\GL_{p-1}(\Z))=(\Q_p)^{\vert \Cl(\mathbb{Q}(\zeta_p) \vert \cdot 2^{\frac{p-5}{2}}}.\]
These formulas combined with L\"uck's theorem (Theorem \ref{thm:Wolfgang}) recover a calculation of Adem (\cite[Example 6.7]{Adem1993} and \cite[Example 4.5]{Adem1992}). We see that the group $\GL_{p-1}(\Z)$ does not satisfy weak duality in $p$-adic $K$-theory and the Farrell--Tate $K$-theory Euler characteristic vanishes. 

\subsection{The symplectic group \texorpdfstring{$\SP_{p-1}(\Z)$}{Sp p-1(Z)} for a prime \texorpdfstring{$p \geq 5$}{p at least 5}}

It follows from \cite[Section 2]{Ash89} and \cite{LevittNicolas98} that any non-trivial $p$-subgroup of $\Sp_{p-1}(\Z)$ is isomorphic to $\Z/p$. The paper \cite{SjerveYang97} shows that the conjugacy classes of the elements of order $p$ in $\Sp_{p-1}(\Z)$ are in bijection with the set of equivalence classes of the pairs $(I, a)$, where $I$ is an integral ideal in $\Z[\zeta_p]$ and $I \cdot \overline{I}=(a)$, where $a=\overline{a}$. It follows form \cite[Theorem 3]{SjerveYang97} that the number of such equivalence classes is equal to $2^{\frac{p-1}{2}}h_p^{-}$, where 
\[h_p^{-}=\frac{\vert \Cl(\mathbb{Q}(\zeta_p) \vert }{\vert \Cl(\mathbb{Q}(\zeta_p+\zeta_p^{-1}))\vert}\] 
is the relative class number. By \cite[Section 3.2]{Busch2002} and \cite{Brown74}, we know that for any element $A$ of order $p$, the centraliser of $A$ is finite and
\[C\langle A \rangle \cong \Z/p \times \Z/2.\] 
As a consequence by Proposition \ref{prop:Tate K-theory}, we get
\[\widehat{K_p}^1(B\Sp_{p-1}(\Z))=0,\]
and
\[\widehat{K_p}^0(B\Sp_{p-1}(\Z)) \cong (\Q_p)^{2^{\frac{p-1}{2}}h_p^{-}}.\]
In particular, $\Sp_{p-1}(\Z)$ satisfies weak duality in $p$-adic $K$-theory.

\subsection{The mapping class group \texorpdfstring{$\MCG(\Sigma_{\frac{p-1}{2}})$}{} for a prime \texorpdfstring{$p \geq 5$}{p at least 5}}

Let $\MCG(\Sigma_{\frac{p-1}{2}})$ be the mapping class group of the closed oriented surface of genus $\frac{p-1}{2}$. The papers \cite{Broughton90, JiWolpert2010, Mislin2010} show that mapping class groups admit a finite model for $\underline{E}G$. 

It follows from \cite[Chapter III]{Xia90} that any non-trivial finite $p$-subgroup of $\MCG(\Sigma_{\frac{p-1}{2}})$ is isomorphic to $\Z/p$. By \cite[Section 3.1]{Xia90}, the number of conjugacy classes of order $p$ elements is equal to
\[\frac{(p+1)(p-1)}{6}.\]
Additionally it follows from \cite[Section 3.2]{Xia90} that the centralisers of order $p$ elements are finite. Hence, by Proposition \ref{prop:Tate K-theory}
\[\widehat{K_p}^1(B\MCG(\Sigma_{\frac{p-1}{2}}))=0\]
and 
\[\widehat{K_p}^0(B\MCG(\Sigma_{\frac{p-1}{2}})) \cong (\Q_p)^{\frac{(p+1)(p-1)}{6}}.\]
In particular, $\MCG(\Sigma_{\frac{p-1}{2}})$ satisfies weak duality in $p$-adic $K$-theory.

\subsection{Amalgamated products of finite groups} Interesting examples arise by considering amalagamated products of finite groups. By Bass--Serre theory \cite{Serre2003}, we know that any such group admits a finite $\underline{E}G$, its Bass--Serre covering tree. Consider
\[G= (\Z/p \rtimes \Z/(p-1)) \ast_{\Z/p}  (\Z/p \rtimes \Z/(p-1) ),\]
where $p$ is an odd prime and $\Z/(p-1)$ is acting as the automorphism group of $\Z/p$. Then $G$ has only one conjugacy class of order $p$ elements and its centraliser is isomorphic to
\[\Z/p \times F_{p-2},\]
where $F_{p-2}$ is the free group on $p-2$ generators. 
By Proposition \ref{prop:Tate K-theory}, we get
\[\widehat{K_p}^0(BG)=\Q_p,\]
and 
\[\widehat{K_p}^1(BG)=\Q_p^{p-2}.\]
This group does not satisfy weak duality in $p$-adic $K$-theory. The Farrell--Tate $K$-theory Euler characteristic is equal to $3-p$ and goes to $-\infty$ when the prime increases. We note though that the group $G$ is rationally acyclic.

To summarise, we saw that the groups $\SL_3(\Z)$, $\Sp_{p-1}(\Z)$, $\MCG(\Sigma_{\frac{p-1}{2}})$ satisfy weak duality, while $\widehat{K_p}^0(BG)$ can be an arbitrarily large finite dimensional $\Q_p$-vector space. The group $\GL_{p-1}(\Z)$ provides an example where $\widehat{K_p}^0(BG)$ and $\widehat{K_p}^1(BG)$ have same dimension (increasing with $p$) and hence the Farrell--Tate $K$-theory Euler characteristic vanishes. Aside from $\SL_3(\Z)$ which is rationally acyclic by \cite{Soule78}, the rational cohomologies of these groups are not fully computed. The last example shows that $\widehat{K_p}^1(BG)$ can be arbitrarily large while $\widehat{K_p}^0(BG)$ is as small as possible and $G$ is rationally acyclic. Below we will show that $\Out(F_{p+1})$ provides a similar example to the latter but in this case the rational cohomology is very interesting and far away from being computed.

\section{Elements of order \texorpdfstring{$p$}{p} and their centralisers} \label{section: Centraliser of Phi} 

In this section we study conjugacy classes of order $p$ elements in $\Out(F_n)$ when $p-1 \leq n \leq 2p-3$ and their centralisers. The results when $n \neq p+1$ are well known and can be found in \cite{Chen97, GMV98}. By \cite[Corollary 2.2]{GMV98} the group $\Out(F_n)$ is $p$-periodic for $p-1 \leq n \leq 2p-3$, i.e. it does not contain $\Z/p \times \Z/p$ as a subgroup. Additionally, the natural map $\Out(F_n) \to \GL_n(\Z)$ has torsion-free kernel and there are no order $p^2$ elements in $\GL_n(\Z)$ in this range  \cite{LevittNicolas98}. Hence any finite $p$-subgroup 
of $\Out(F_n)$ for $p-1 \leq n \leq 2p-3$ is isomorphic to $\Z/p$. 

All the graphs we consider in this section are connected. We establish the following conventions for our actions. (Note that as soon as one of these is decided, the others follow.) \begin{enumerate}
	\item $F_n$ acts on trees on the left;
	\item $\Aut(F_n)$ acts on $F_n$ on the left;
	\item $\Aut(F_n)$ (hence $\Out(F_n)$) acts on $\CV_n$ on the right.
\end{enumerate}

We need the following theorem, the $\Out(F_n)$ realisation theorem, which was proved independently by Culler \cite{Culler1984}, Khramtsov \cite{Khramtsov1987} and Zimmermann \cite{Zimmerman1981Realisation}.
\begin{thm}\label{thm: realisation}
	Every finite subgroup of $\Out(F_n)$ is realised by a group of isometries of a finite graph; every finite subgroup of $\Aut(F_n)$ is realised by such a group of isometries preserving some basepoint.
\end{thm}

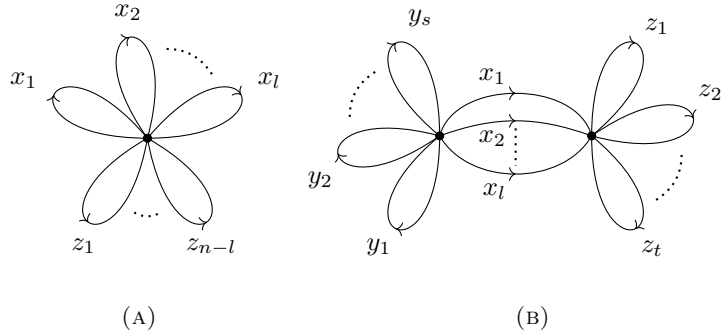
\begin{figure}[h]
\begin{subfigure}{.3\linewidth}
    \begin{tikzpicture}
	\coordinate (v0) at (0,0);
	\draw[fill,black] (v0) circle (1.5pt);
	\draw[middlearrow={.5}{20}{>}{$x_1$}{above left =2pt, rotate=-65}] (v0) to[out=180, in=130,loop,distance=2cm] (v0);
	\draw[middlearrow={.5}{20}{>}{$x_2$}{above=3pt,rotate=-10}] (v0) to[out=125, in=75,loop,distance=2cm] (v0);
    \draw[middlearrow={.5}{20}{>}{$x_l$}{above right=2pt,rotate=65}] (v0) to[out=50, in=0,loop,distance=2cm] (v0);
    \draw[dotted,bend right,thick] (45:1.2) to (80:1.2);

    \draw[middlearrow={.5}{-20}{>}{$z_1$}{below =3pt, rotate=35}] (v0) to[out=210, in=260,loop,distance=2cm] (v0);
    \draw[middlearrow={.5}{-20}{>}{$z_{n-l}$}{below=3pt, rotate=-35}] (v0) to[out=280, in=330,loop,distance=2cm] (v0);
    \draw[dotted,bend right,thick] (260:1) to (280:1);
\end{tikzpicture}
    \caption{\label{fig:rose}}
\end{subfigure}
\begin{subfigure}{.5\linewidth}
    \begin{tikzpicture}
	\coordinate (v0) at (0,0);
	\draw[fill,black] (v0) circle (1.5pt);
	\coordinate (v1) at (2,0);
	\draw[fill,black] (v1) circle (1.5pt);
	\draw[middlearrow={.5}{20}{>}{$y_1$}{above=15pt, rotate=-155}] (v0) to[out=270, in=220,loop,distance=2cm] (v0);
	\draw[middlearrow={.5}{20}{>}{$y_2$}{left=10pt,rotate=-100}] (v0) to[out=215, in=165,loop,distance=2cm] (v0);
    \draw[middlearrow={.5}{20}{>}{$y_s$}{above right=2pt,rotate=-25}] (v0) to[out=140, in=90,loop,distance=2cm] (v0);
    \draw[dotted,bend right,thick] (135:1.2) to (170:1.2);

    \draw[middlearrow={.5}{20}{>}{$z_1$}{above right, rotate=25}] (v1) to[out=90, in=40,loop,distance=2cm] (v1);
	\draw[middlearrow={.5}{20}{>}{$z_2$}{right=-10pt,rotate=80}] (v1) to[out=35, in=-15,loop,distance=2cm] (v1);
    \draw[middlearrow={.5}{20}{>}{$z_t$}{above=15pt,rotate=155}] (v1) to[out=-40, in=-90,loop,distance=2cm] (v1);
    \draw[dotted,bend right,thick] ($(v1)+(-45:1.2)$) to ($(v1)+(-10:1.2)$);

    \draw[middlearrow={.5}{0}{>}{$x_1$}{above left}] (v0) to[out=75, in=105] (v1);
    \draw[middlearrow={.5}{0}{>}{$x_2$}{below left}] (v0) to[out=20, in=160] (v1);
    \draw[middlearrow={.5}{0}{>}{$x_l$}{below left}] (v0) to[out=-60, in=-120] (v1);
    \draw[dotted,thick] (1,0.1) -- (1,-0.4);
\end{tikzpicture}
    \caption{\label{fig:theta}}
\end{subfigure}
\caption{Graphs with isometries realising $R_l$ and $\theta_{st}$
\label{fig:rose_and_theta}}
\end{figure}

First we review the rose and theta type elements $R_l$ and $\theta_{st}$, where $2 \leq l \leq n$ and $0\leq s \leq t \leq n$ and $s+t+l-1=n$. They are described algebraically for instance in \cite{Chen97}; geometrically they can be realised by isometries of the graphs in Figure~\ref{fig:rose_and_theta}. The element $R_l$ is realised on the rose in (A), where it cyclically permutes the edges labelled $x_i$ and leaves the rest fixed. The element $\theta_{st}$ is realised on the graph in (B), by an isometry cyclically permuting the edges labelled $x_i$ and leaving the others fixed. In this case there is a choice of basepoint to be made when defining an element of $\Aut(F_n)$: the convention is that the basepoint for $\theta_{st}$ is chosen to be the one on the right hand side of the figure. If $s\neq t$ these choices give elements $\theta_{st}$ and $\theta_{ts}$ that are not conjugate in $\Aut(F_n)$ but whose images become conjugate in $\Out(F_n)$. These elements have order $l$ and their centralisers in $\Out(F_n)$ are computed in \cite[Section 2.2 and Section 2.3]{Chen97}:

\begin{prop} \label{prop: rosethetacentralisers} Let $2 \leq l \leq n$ and $0\leq s \leq t \leq n$ and $s+t+l-1=n$. Then we have
\[C \langle R_l \rangle \cong \Z/l \times ((F_{n-l} \rtimes \Aut(F_{n-l})) \rtimes \Z/2)\]
and 
\[C \langle \theta_{st} \rangle \cong \Z/l \times ((\Aut(F_s) \times \Aut(F_t)) \rtimes \Z/2^{\delta_{st}}),\]
where $\delta_{st}$ is the Kronecker delta. 
\end{prop}

Here $\Z/2$ acts on $F_{n-l} \rtimes \Aut(F_{n-l})$ by sending $(x,\sigma)$ to $(x^{-1}, \alpha_x \sigma)$, where $x \in F_{n-l}$, and $\sigma \in \Aut(F_{n-l})$, and $\alpha_x(t)=xtx^{-1}$ (see e.g. \cite[page 32]{Chen97}), and when $t=s$, then $\Z/2$ acts on $\Aut(F_t) \times \Aut(F_t)$ by flipping the factors. 

By construction, the rose and theta type elements of order $l$ come from order $l$ elements in $\Aut(F_n)$. In fact in the range $p-1 \leq n \leq 2p-3$ any order $p$ element of $\Out(F_n)$ is conjugate to a theta or rose type element except when $n=p+1$ \cite[Proposition 3.1.2]{Chen97}. The rest of the section is devoted to investigating this last case. 

We show in Corollary \ref{cor: strange order p element} that up to conjugacy there is exactly one element of order $p$ in $\Out(F_{p+1})$ which does not lift to an order $p$ element in $\Aut(F_{p+1})$ (every other order $p$ element in $\Out(F_{p+1})$ is conjugate to a rose or theta type element). We denote this element by $\Phi$. The element $\Phi$ and its centraliser seem not to be well-studied. Its existence and uniqueness has been recently independently proved by Bridson and Piwek \cite[Proposition 7.1]{BridsonPiwek}. We compute the centraliser of $\Phi$ which seems to be a new result. In particular, the calculations of \cite{Chen97} and \cite{GloverMislin2000} do not give any information about the centraliser. 

Our arguments take place in the (reduced) equivariant spine, introduced by Krstic--Vogtmann \cite{KrsticVogtmann1993} and subsequently studied for instance by Bestvina--Feighn--Handel \cite{BestvinaFeighnHandel2023McCoolGroups}. For a finite order element $\Phi$ of $\Out(F_n)$, or more generally a finite subgroup, this is a certain subspace of the outer space $\CV_n$. It is a subcomplex of the spine of $\CV_n$, consisting of fixed points of $\Phi$, so it is (like the spine itself) a simplicial complex. It is not an $\Out(F_n)$-invariant subspace, but it is preserved by the centraliser and normaliser in $\Out(F_n)$ of $\Phi$. It plays two dual roles in our proof. The first is in solving the conjugacy problem, following Krstic--Lustig--Vogtmann's algorithm \cite{KrsticLustigVogtmann}; two finite order elements of $\Out(F_n)$ are conjugate if they are induced by the same automorphism of the same (unmarked) graph; understanding the options amounts to understanding the vertex orbits of the centraliser action on the equivariant spine. In particular, we will use the elementary equivariant expansions and collapses introduced by Krstic--Vogtmann to enable us to find identical representatives. The second is in understanding the centraliser through its action on this spine: while strictly speaking we work in the spine of outer space corresponding to a different free group, work of Lyman \cite{LymanIsometries} implies that there is an equivariant isometry between these spines.

\begin{prop}
	\label{prop:nice_representative}
	Suppose $\Phi$ is an element of $\Out(F_n)$ of order $p$ a prime, and is not represented by an element of $\Aut(F_n)$ of order $p$. Then $\Phi$ is realised by an automorphism $f$ of a graph $\Gamma$ having the following properties: \begin{enumerate}		
		\item The automorphism $f$ acts freely on $\Gamma$;
		\item There is a single orbit of vertices;
		\item There is an orbit of edges forming a $p$-cycle, and (choosing the appropriate orientation) with $f(\iota(e))=\tau(e)$, where $\iota(e)$ denotes the source vertex and $\tau(e)$ denotes the target vertex;
		\item All edges outside this orbit are loops at a vertex, satisfying $\iota(e)=\tau(e)$.
	\end{enumerate}
\end{prop}

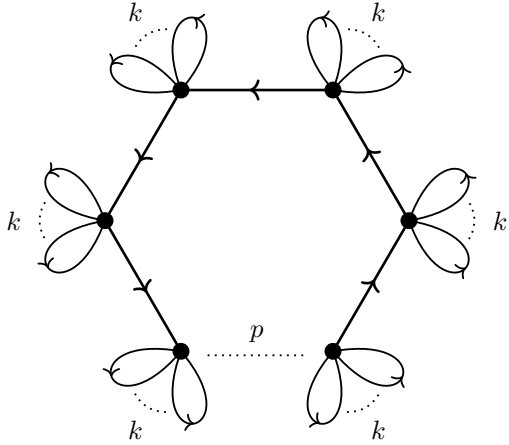
\begin{figure}[h!]
\begin{tikzpicture}[line width=0.7pt]
	\foreach \x in {-60, 0,60,120,180,240}{
		\draw[middlearrow={.5}{-20}{>}{}{}] (\x:2) to[out=\x-70, in=\x-10,loop,distance=1.5cm] (\x:2);
		\draw[middlearrow={.5}{-20}{>}{}{}] (\x:2) to[out=\x+10, in=\x+70,loop,distance=1.5cm] (\x:2);
		\draw[dotted, bend right]  (\x-5:2.8) to (\x+5:2.8);
		\node[inner sep=2,circle,draw,fill] at (\x:2){};
		\node at (\x:3.2) {$k$};}
	\foreach \x in {-60, 0,60,120,180}{
		\draw[line width = 1pt] (\x:2) edge[middlearrow={0.55}{0}{>}{}{}] (\x+60:2);}
	\draw[dotted] (250:1.9) to (290:1.9); \node at (270:1.5) {$p$};
\end{tikzpicture}
\caption{The outer automorphisms considered in Proposition~\ref{prop:nice_representative} are all realised by a graph automorphism rotating this $\Gamma$ by $2\pi/p$.}
\end{figure}

A key step is the existence of the following ``equivariant slide'' move in the situations we are considering.

\begin{lem}
	\label{lem:equivariant_slide}
	Suppose $f$ is an automorphism of a graph $\Gamma$ so that the action is free, with one orbit of vertices, and that edges $s$ and $t$ are in distinct orbits and form a (coherently oriented) path of length 2, meeting at $v=\tau(s)=\iota(t)$. Then there is an equivariant Whitehead move \emph{sliding} $s$ along $t$.
\end{lem}


That is, the effect of the move is that -- after relabelling -- the graph and action are unchanged except now $\tau(s)=\tau(t)$ (and of course this behaviour is duplicated across the orbit).

\begin{proof}
	First do an equivariant expansion at $v$, introducing an edge labelled $t'$ and choosing to attach $t$ coherent with $t'$, $s$ incoherent with $t'$ and all other edges at $v$ to the initial vertex of $t'$.
	
	Now collapse $t$: after relabelling, $t'$ joins the vertices joined by $t$ in the original graph, but $s$ joins $\iota(s)$ to $\tau(t')$ rather than $\iota(t)$ as originally. Figure~\ref{fig:equivariant_slide} shows the steps; note however that there is no requirement for any of the vertices to be distinct.
\end{proof}

\begin{figure}
	\begin{subfigure}[t]{.4\linewidth}
		\begin{tikzpicture}
		\node[circle, inner sep = 2pt, draw=black, fill](v1) at (0,0) {}; 
		\node[circle, inner sep = 2pt, draw=black, fill](v2) at (2,0) {}; 
		\node[circle, inner sep = 2pt, draw=black, fill](v3) at (4,0) {}; 
		\draw[middlearrow={.55}{0}{>}{$s$}{above}] (v1) to (v2);
		\draw[middlearrow={.55}{0}{>>}{$t$}{above}] (v2) to (v3);
		\foreach \v in {v1,v2,v3}{
			\draw (\v) to ($(\v)+(120:0.8)$);
			\draw (\v) to ($(\v)+(100:0.8)$);
			\draw[dotted] ($(\v)+(90:0.6)$) to ($(\v)+(65:0.6)$);
			\draw (\v) to ($(\v)+(60:0.8)$);}
		\draw[middlearrow={.55}{0}{>}{}{}] ($(v1)+(245:0.8)$) to (v1);
		\draw[middlearrow={.55}{0}{>}{}{}] (v2) to ($(v2)+(315:0.8)$);
		\draw[middlearrow={.55}{0}{>}{}{}] ($(v3)+(225:0.8)$) to (v3);
		\draw[middlearrow={.55}{0}{>}{}{}] (v3) to ($(v3)+(335:0.8)$);
		
		\draw[middlearrow={.55}{0}{>>}{}{}] ($(v1)+(195:0.8)$) to (v1);		
		\draw[middlearrow={.55}{0}{>>}{}{}] (v1) to ($(v1)+(315:0.8)$);
		\draw[middlearrow={.55}{0}{>>}{}{}] ($(v2)+(225:0.8)$) to (v2);
		\draw[middlearrow={.55}{0}{>>}{}{}] (v3) to ($(v3)+(295:0.8)$);
		
	\end{tikzpicture}
		\caption{Initial state}
	\end{subfigure}
    
	\begin{subfigure}[t]{.4\linewidth}
		\begin{tikzpicture}
		\node[circle, inner sep = 2pt, draw=black, fill](v1p) at (-0.5,0) {}; 
		\node[circle, inner sep = 2pt, draw=black, fill](v1) at (0.2,0) {}; 
		\node[circle, inner sep = 2pt, draw=black, fill](v2p) at (2,0) {}; 
		\node[circle, inner sep = 2pt, draw=black, fill](v2) at (2.7,0) {}; 
		\node[circle, inner sep = 2pt, draw=black, fill](v3p) at (4,0) {}; 
		\node[circle, inner sep = 2pt, draw=black, fill](v3) at (4.7,0) {}; 
		\draw[middlearrow={.7}{0}{>>>}{}{}] (v1p) to (v1);
		\draw[middlearrow={.7}{0}{>>>}{$t'$}{above}] (v2p) to (v2);
		\draw[middlearrow={.7}{0}{>>>}{}{}] (v3p) to (v3);
		\draw[middlearrow={.55}{0}{>>}{$t$}{above}] (v2) to (v3p);
		\draw[bend right, middlearrow={.55}{0}{>}{$s$}{above}] (v1p) to (v2);
		\foreach \v in {v1p,v2p,v3p}{
			\draw (\v) to ($(\v)+(120:0.8)$);
			\draw (\v) to ($(\v)+(100:0.8)$);
			\draw[dotted] ($(\v)+(90:0.6)$) to ($(\v)+(65:0.6)$);
			\draw (\v) to ($(\v)+(60:0.8)$);}
		
		\draw[middlearrow={.55}{0}{>}{}{}] ($(v1)+(45:0.8)$) to (v1);
		\draw[middlearrow={.55}{0}{>}{}{}] (v2p) to ($(v2p)+(150:0.8)$);
		\draw[middlearrow={.55}{0}{>}{}{}] ($(v3)+(255:0.8)$) to (v3);
		\draw[middlearrow={.55}{0}{>}{}{}] (v3p) to ($(v3p)+(285:0.8)$);
		
		\draw[middlearrow={.55}{0}{>>}{}{}] ($(v1p)+(195:0.8)$) to (v1p);		
		\draw[middlearrow={.55}{0}{>>}{}{}] (v1) to ($(v1)+(65:0.8)$);
		\draw[middlearrow={.55}{0}{>>}{}{}] ($(v2p)+(170:0.8)$) to (v2p);
		\draw[middlearrow={.55}{0}{>>}{}{}] (v3) to ($(v3)+(335:0.8)$);
	\end{tikzpicture}
		\caption{After the expansion}
	\end{subfigure}

	\begin{subfigure}[b]{.4\linewidth}
		\begin{tikzpicture}
		\node[circle, inner sep = 2pt, draw=black, fill](v1p) at (0,0) {}; 
		\node[circle, inner sep = 2pt, draw=black, fill](v2p) at (2,0) {}; 
		\node[circle, inner sep = 2pt, draw=black, fill](v3p) at (4,0) {}; 
		\draw[middlearrow={.6}{0}{>>>}{$t'$}{above}] (v2p) to (v3p);
		\draw[bend right, middlearrow={.55}{0}{>}{$s$}{below}] (v1p) to (v3p);
		\foreach \v in {v1p,v2p,v3p}{
			\draw (\v) to ($(\v)+(120:0.8)$);
			\draw (\v) to ($(\v)+(100:0.8)$);
			\draw[dotted] ($(\v)+(90:0.6)$) to ($(\v)+(65:0.6)$);
			\draw (\v) to ($(\v)+(60:0.8)$);}
		
		\draw[middlearrow={.55}{0}{>}{}{}] ($(v1p)+(215:0.8)$) to (v1p);
		\draw[middlearrow={.55}{0}{>}{}{}] (v2p) to ($(v2p)+(150:0.8)$);
		\draw[middlearrow={.55}{0}{>}{}{}] ($(v2p)+(190:0.8)$) to (v2p);
		\draw[middlearrow={.55}{0}{>}{}{}] (v3p) to ($(v3p)+(285:0.8)$);
		
		\draw[middlearrow={.7}{0}{>>>}{}{}] (v1p) to ($(v1p)+(0:0.8)$);
		\draw[middlearrow={.7}{0}{>>>}{}{}] ($(v1p)+(195:0.8)$) to (v1p);
		\draw[middlearrow={.7}{0}{>>>}{}{}] (v3p) to ($(v3p)+(0:0.8)$);
		\draw[middlearrow={.7}{0}{>>>}{}{}] ($(v2p)+(170:0.8)$) to (v2p);
		
	\end{tikzpicture}
		\caption{After the collapse}
	\end{subfigure}
	\caption{Schematic for the Whitehead move described in Lemma~\ref{lem:equivariant_slide}.}
	\label{fig:equivariant_slide}
\end{figure}
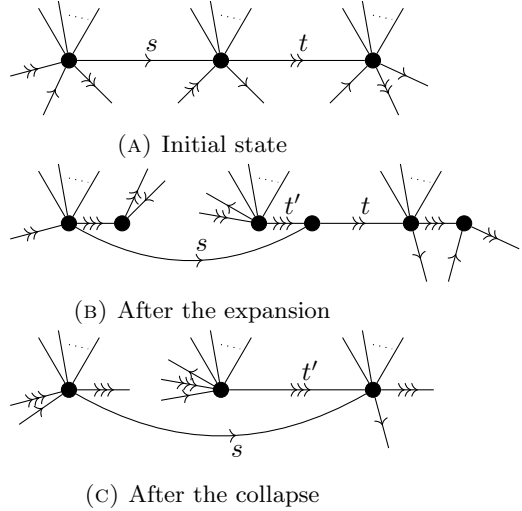


\begin{proof}[Proof of Propsition~\ref{prop:nice_representative}]
	By Theorem \ref{thm: realisation}, the element $\Phi$ is induced by an isometry of a finite (marked) graph $\Gamma$. Denote this isometry by $f$: we will gradually alter $f$ and $\Gamma$ until they satisfy all the required properties. 
    
    If a graph automorphism inducing $\Phi$ fixes a vertex, then it gives rise to a finite order element of $\Aut(F_n)$. So any graph automorphism inducing an element satisfying the hypotheses cannot have a fixed point. (If $p=2$, one might be concerned that the fixed point was the midpoint of an edge; this is equally disallowed by the assumption that there cannot be a fixed base point.) Since cyclic groups of prime order have no proper subgroups, all stabilisers are trivial, and the action is free. In particular, $\Gamma$ must have $l$ orbits each of $p$ vertices. Suppose $l$ is at least 2. Then in order for $\Gamma$ to be connected there is an edge $e$ joining two vertices in different orbits. The orbit of $e$ consists of $p$ edges with disjoint adjacent vertices, so is an equivariant forest, and we can collapse it. Repeating this, we end up with a graph $\Gamma$ supporting an automorphism realising $\Phi$ with only a single orbit of $p$ vertices.
	
	We now describe a series of equivariant slides which will arrange for the edges to satisfy the final two claims. Picking any vertex $v$, we first arrange for an edge joining $v$ to $f(v)$. Since $\Gamma$ is connected, there is a reduced path $\gamma=e_1e_2 \dots e_n$ joining these two vertices. If it has length 1 we are done; so we will describe a procedure to reduce its length. If any $e_i$ and $e_{i+1}$ in this path are in different orbits then we may apply the equivariant slide of Lemma~\ref{lem:equivariant_slide} to slide $e_i$ across $e_{i+1}$, resulting in a graph $\Gamma'$, again with one orbit of vertices and a free action of $f$, but now with a shorter path $\gamma' = e_1 \dots e'_{i} e_{i+2} \dots e_n$ joining $v$ to $f(v)$. (The reduction might occur in more than one place simultaneously, in which case the length of the new path would be less than $n-1$.)
	
	So we can assume that $\gamma$ is composed of edges which are all in the same orbit. The concatenation $\gamma f(\gamma) \dots f^m(\gamma)$  will again be composed of edges in the same orbit and joins $v$ to $f^{m+1}(v)$; running $m=0$ up to $m=p-1$ implies that there are paths composed of edges in the same orbit joining $v$ to every other vertex. Since the edge orbits are also of cardinality $p$, this can only happen if this orbit forms a $p$-cycle, visiting (in some order) each of $v, f(v), \dots f^{p-1}(v)$.
	
	Let $e$ be the edge in this cycle originating at $v$. Now take any other edge $d$ at $v$, and observe that its endpoint lies somewhere on the cycle consisting of the edges $e,f(e),f^2(e), \dots, f^{p-1}(e)$ (in some order), since this visits all vertices exactly once. Repeatedly applying Lemma~\ref{lem:equivariant_slide} to slide $d$ over edges in the orbit of $e$ will at some stage (at most $p-1$ slides are required, or $(p-1)/2$ if one picks the opposite orientation on the orbit of $e$) produce a $d'$ joining $v$ to $f(v)$, as we required.
	
	Fixing an orbit of edges $f^i(e)$ satisfying Property (3) of the statement, we may slide any other orbit of edges along this cycle until $\tau(e)=\iota(e)$, eventually arranging for all other orbits to be loops, as claimed by Property (4). 
\end{proof}

We now deduce certain algebraic properties of $\Phi$, its centraliser $C\langle \Phi \rangle $ and its normaliser $N\langle \Phi \rangle$ from the realisation given by Proposition~\ref{prop:nice_representative}.

\begin{cor}\label{cor: strange order p element}
	Suppose $\Phi$ is an element of $\Out(F_n)$ of order $p$ a prime, and is not represented by an element of $\Aut(F_n)$ of order $p$. Then: \begin{enumerate}
		\item $n \equiv 1 \mod p$;
		\item $\Out(F_n)$ contains one conjugacy class of elements fitting this description;
		\item The centraliser $C \langle\Phi \rangle$ maps with finite kernel onto a finite index subgroup of $\Out(F_m)$, with $m=1+(n-1)/p$;
		\item The quotient $N\langle \Phi \rangle / C\langle \Phi \rangle$ is isomorphic to $\Z/(p-1)$.
	\end{enumerate}
\end{cor}

Claims (1) and (2) were independently obtained as part of \cite[Proposition 7.1]{BridsonPiwek}, in the context of studying the free-by-cyclic groups defined by such an automorphism.

\begin{proof}
	Let $\Gamma$ be the graph and $f$ the automorphism realising $\Phi$ given by Proposition~\ref{prop:nice_representative}. Then:
	\begin{enumerate}
		\item The rank of $\Gamma$ is $p(k+1)-p+1$, where $k$ is the number of loops at each vertex.
		\item Since by Proposition~\ref{prop:nice_representative} any outer automorphism with the given properties can be represented on the same graph with the same action, \cite[Theorem 2.4]{KrsticLustigVogtmann} implies that they are in the same conjugacy class.
		\item Consider the preimage $\widehat{\langle \Phi \rangle}$ in $\Aut(F_n)$ of the subgroup generated by $\Phi$. This is a virtually free group without any assumption on $\Phi$ beyond finite order, but in the set up here it is actually free. To see this, we use the fact that it acts on the universal covering tree $T$ of $\Gamma$ (this is a consequence of the proofs of Theorem~\ref{thm: realisation}, which begin by observing that $\widehat{\langle \Phi \rangle}$ is virtually free and hence acts on a tree by \cite{KPSvirtfree}; interpreting the alterations to $\Gamma$ and $f$ as equivariant alterations to this tree imply the claim), and this action is free since this is true both of the action of $\Inn(F_n) \cong F_n$ on $T$ and of the induced action of $\widehat{\langle \Phi \rangle}/\Inn(F_n) \cong \langle \Phi \rangle$ on $T/\Inn(F_n) \cong \Gamma$.
		
		This means that $\widehat{\langle \Phi \rangle}$ is free. Its rank is the rank of $\Gamma/\langle \Phi \rangle$, which is $m= k+1$; since $n= pk+1$ this is as in the statement. We have the following short exact sequence		
		\[ 1 \to F_n \to F_m \cong \widehat{\langle \Phi \rangle} \to \Z/p \to 1.\]
		
		It follows from \cite[Proposition 5.4.2]{Andrew2022} that the (finite index) subgroup of $\Aut(F_m)$ preserving the finite index subgroup $F_n$ and inducing the trivial automorphism on the finite quotient is isomorphic to the subgroup of $\Aut(F_n)$ commuting with $\Phi \in \Out(F_n)$ up to inner automorphisms, which is the preimage of $C \langle \Phi \rangle $. Passing this isomorphism to the quotients, we see that $C \langle \Phi \rangle $ must map with finite kernel (the subgroup $\langle \Phi \rangle$) to a finite index subgroup of $\Out(F_m)$.
		\item Notice that since $\Phi$ is of order $p$ and not represented by a finite order automorphism, the same is true of all its non-trivial powers. By part (2), this means that $\Phi$ is conjugate to each $\Phi^k$ with $k \in \{1 \dots p-1\}$. These conjugating elements are contained in the normaliser of $\langle \Phi \rangle$, and so we see that this normaliser induces the whole of $\Aut(\langle \Phi \rangle)$, which is isomorphic to $\Z/(p-1)$. The kernel of the map $N\langle \Phi \rangle \to \Aut(\langle \Phi \rangle)$ is exactly the centraliser of $\Phi$. \qedhere
	\end{enumerate}
\end{proof}

If $N$ is a normal subgroup of $F_n$, then it is $\Inn(F_n)$-invariant, and we so can consider the $\Out(F_n)$ action on the set of normal, index $p$ subgroups of $F_n$. If, as in this case, the quotient is abelian, then $g^{-1}tgN = g^{-1}N tN gN = tN$, so the inner automorphisms also preserve all of the cosets of these subgroups, and there is an induced $\Out(F_n)$ action on these cosets. In the next lemma we consider a stabiliser of some normal subgroup $N$ and its cosets, and begin to understand its action on Culler--Vogtmann space.

\begin{lem}
	\label{lem:counting_orbits}
	Suppose $x$ is a point in $CV_n$, where $n \geq 2$, and $A$ is the finite index subgroup of $\Out(F_n)$ preserving some normal subgroup $N$ of $F_n$, which has prime index $p$, as well as its cosets. Then the following sets have the same finite cardinality: \begin{itemize}
		\item the $A$-orbits of points in the $\Out(F_n)$-orbit of $x$;
		\item the $\Out(F_n)_x$-orbits of pairs $(N, tN)$ where $N$ is an index $p$ normal subgroup of $F_n$ and $tN \neq N$;
		\item the $\Out(F_n)_x$-orbits of non-trivial maps $F_n \to \Z/p$.
	\end{itemize}
\end{lem}

\begin{proof}
	A non-trivial map from $F_n$ to $\Z/p$ is uniquely determined by its kernel and by identifying the pre-image of 1, which is to say the data $(N,tN)$. This gives the bijection between the second and third items. Since these actions are of a finite group on a finite set, the cardinality is finite. (There is a change of side of the action, but there is a correspondence between the orbits: precomposing a map $F_n \to \Z/p$ with $\varphi$ is equivalent to sending $(N,tN)$ to $(\varphi^{-1}(N),\varphi^{-1}(tN))$.) 
	
	Now we claim that, for the action of $\Out(F_n)$, there is only one such orbit. Through the lens of non-trivial maps $F_n \to \Z/p$, this is equivalent to the claim that there is a single \emph{Nielsen equivalence class} (orbit under the action of $\Aut(F_n)$) of generating sets of $\Z/p$ with cardinality $n \geq 2$. This follows from the Euclidean algorithm; a proof (in the more general setting of finite abelian groups) is given in \cite{DiaconisGraham1999}. 
    
    Since $N$ is normal and prime index, the quotient group $F_n/N$ is isomorphic to $\Z/p$ and an automorphism of this quotient is determined by the image of any non-trivial element. Hence if $\Phi$ preserves $N$, it preserves all the cosets of $N$ if and only if it preserves any non-trivial coset. So $A$ is exactly equal to the stabiliser of some (hence every) pair $(N,tN)$. 
    
	In general, if $G$ acts (on the left) on some set with $G_x$ the stabiliser of some element $x$, and $H$ is a subgroup of $G$, then the double cosets $HgG_x$ index the $H$-orbits that partition the $G$-orbit of $x$. If the $G$ action was on the right, then they are indexed by $G_x g H$ instead. We have two actions of $\Out(F_n)$ on different sides: the left action on conjugacy classes in $F_n$, inducing a transitive left action on the pairs $(N,tN)$, and a right action on $\CV_n$. Hence the double cosets $\Out(F_n)_x \Psi A$ index both the $A$-orbits of points in the $\Out(F_n)$-orbit of $x$, and the $\Out(F_n)_x$-orbits of pairs $(N, tN)$. So the cardinality of the first item agrees with the other two.
\end{proof}

We now specialise to the case of $n=p+1$, or equivalently $m=2$.

\begin{prop}
	\label{prop:action}
	Suppose $p$ is prime, $n=p+1$, and $\Phi$ is an order $p$ element of $\Out(F_n)$ that is not represented by an element of $\Aut(F_n)$ of order $p$. Then $C \langle \Phi \rangle $ acts on a tree with finite stabilisers, and if $p\geq 5$ the action has $\frac{1}{24}(p-1)(5p+29)$ orbits of vertices and $\frac{1}{4}(p-1)(p+3)$ orbits of edges. If $p=2$, then there are 3 orbits of vertices and 2 of edges, and if $p=3$ there are 4 orbits of vertices and 3 of edges.
\end{prop}

\begin{proof}
	We use the identification of $C \langle \Phi \rangle /\langle \Phi \rangle$ with the subgroup $A$ of $\Out(F_2)$ preserving an index $p$ normal subgroup as well as its cosets. The reduced spine of $\CV_2$ is a tree with a minimal action of $\Out(F_2)$. This action has two orbits of vertices, ``rose-type'' and ``theta-type'', and one orbit of edges, and the stabilisers are finite (we describe them below). Hence there is a minimal action of $C \langle \Phi \rangle $ via its quotient to $A$; the stabilisers for this action are also finite as extensions of finite groups by finite groups.
	
	By Lemma~\ref{lem:counting_orbits} we can calculate the number of $A$ orbits each $\Out(F_2)$ orbit is partitioned into by instead considering the orbits of the point stabilisers on the set of non-trivial maps $F_2 \to \Z/p$. In each case, this is an action of a finite group on a finite set, and so we may use the ``lemma that is not Burnside's'' \cite[Theorem 3.22]{Rotman}. This asserts that the number of orbits is equal to \[\frac{1}{|G|}\sum_{g \in G} |\Fix(g)|, \]
	
	\noindent where $\Fix(g)$ is the set of points fixed by the element $g$. The (non-trivial) maps to $\Z/p$ can be characterised by a pair $(l,m)$ of elements, taken to be the images of a fixed basis $\{a,b\}$ of $F_2$, and such that $(l,m) \neq (0,0)$, and we use this characterisation in what follows.
	
	For the edge, the stabiliser is isomorphic to the Klein four group $D_2$, and choosing an appropriate representative of the orbit it is generated by $\tau$ interchanging $a$ and $b$, and $\sigma_e$ inverting both generators. The situation is as in Table~\ref{t:D_2}: the ``implications'' column records the condition that the map determined by $(l,m)$ is fixed by precomposition with the given automorphism. The ``solutions'' are determined modulo $p$;  for $p=2,3$ there can be extra solutions so these are recorded separately.
	
	\begin{table}[h!]
		\begin{tabular}{|c|c|c|c|c|}
			\hline
			&& \multicolumn{3}{c|}{Solutions} \\
			Automorphism & Implications & mod 2 & mod 3 & mod $p$ \\ \hline
			$\Id$ & none & 3 & 8 & $p^2-1$ \\
			$\sigma_e$ & $(l,m) = (-l,-m)$ & 3 & 0 & 0 \\
			$ \tau$ & $ (l,m) = (m,l)$ & 1 & 2 & $p-1$ \\
			$ \tau \sigma_e$ & $(l,m) = (-m,-l)$ & 1 & 2 & $p-1$ \\ \hline
		\end{tabular}
        \caption{Edge orbits.}
        \label{t:D_2}
	\end{table}
	
	Summing, there are 2 orbits of edges if $p=2$, 3 if $p=3$ and in general $\frac{1}{4}(p-1)(p+3)$.
	
	For the rose-type vertices, the stabiliser is isomorphic to $D_4$, with (after choosing a representative) generators $\tau$ interchanging $a$ and $b$ and $\sigma_r$ sending $a \mapsto b \mapsto a^{-1}$. (One can check this is order 4, even as an automorphism.) Note that $\sigma_r^2=\sigma_e$.

	\begin{table}[h!]
		\begin{tabular}{|c|c|c|c|c|}
			\hline
			&& \multicolumn{3}{c|}{Solutions} \\
			Automorphism & Implications & mod 2 & mod 3 & mod $p$ \\ \hline
			$\Id$ & none & 3 & 8 & $p^2-1$ \\
			$\sigma_r$ & $(l,m) = (m, -l)$ & 1 & 0 & 0 \\
			$\sigma_r^2$ & $(l,m) = (-l,-m)$ & 3 & 0 & 0 \\
			$\sigma_r^3$ & $(l,m) = (-m, l)$ & 1 & 0 & 0 \\
			$ \tau$ & $ (l,m) = (m,l)$ & 1 & 2 & $p-1$ \\
			$\tau \sigma_r $ & $(l,m)=(l,-m)$ & 3 & 2 & $p-1$ \\
			$ \tau \sigma_r^2$ & $(l,m) = (-m,-l)$ & 1 & 2 & $p-1$ \\
			$ \tau \sigma_r^3$ & $(l,m) = (-l,m)$ & 3 & 2 & $p-1$ \\ \hline
		\end{tabular}
       \caption{Rose vertex orbits.} 
       \label{t:D_4}
	\end{table}

The situation is shown in Table~\ref{t:D_4}. Summing, there are 2 orbits of rose-type vertices if $p=2,3$ and in general $\frac{1}{8}(p-1)(p+5)$.
	
For the theta-type vertices, the stabiliser is isomorphic to $D_6$ with (after choosing a representative) generators $\tau$ interchanging $a$ and $b$ and $\sigma_t$ sending $a \mapsto b \mapsto ba^{-1}$. This last generator is order 6 only as an outer automorphism, and $\sigma_t^3$ represents the same outer automorphism as $\sigma_e$.
	
	\begin{table}[h!]
		\begin{tabular}{|c|c|c|c|c|}
			\hline
			&& \multicolumn{3}{c|}{Solutions} \\
			Automorphism & Implications & mod 2 & mod 3 & mod $p$ \\ \hline
			$\Id$ & none & 3 & 8 & $p^2-1$ \\
			$\sigma_t$ & $(l,m)=(m,m-l)$ & 0 & 0 & 0 \\
			$\sigma_t^2$ & $(l,m)=(m-l,-l)$ & 0 & 2 & 0 \\
			$\sigma_t^3$ & $(l,m) = (-l,-m)$ & 3 & 0 & 0 \\
			$\sigma_t^4$ & $(l,m)=(-m,l-m)$ & 0 & 2 & 0 \\
			$\sigma_t^5$ & $(l,m)=(l-m,l)$ & 0 & 0 & 0 \\			
			$ \tau$ & $ (l,m) = (m,l)$ & 1 & 2 & $p-1$ \\
			$ \tau \sigma_t$ & $(l,m) = (l,l-m)$ & 1 & 2 & $p-1$ \\
			$ \tau \sigma_t^2$ & $(l,m) = (l-m,-m)$ & 1 & 2 & $p-1$ \\
			$ \tau \sigma_t^3$ & $(l,m) = (-m,-l)$ & 1 & 2 & $p-1$ \\ 
			$ \tau \sigma_t^4$ & $(l,m) = (-l,m-l)$ & 1 & 2 & $p-1$ \\
			$ \tau \sigma_t^5$ & $(l,m) = (m-l,m)$ & 1 & 2 & $p-1$ \\ \hline
		\end{tabular}
        \caption{Theta vertex orbits.}
        \label{t:D_6}
	\end{table}
	
	The situation is shown in Table~\ref{t:D_6}. Summing, there is one orbit of theta-type vertices if $p=2$, 2 orbits if $p=3$ and in general $\frac{1}{12}(p-1)(p+7)$.
	
	Summing the two kinds of vertices, we see that the action of $A$ on the tree has 3 orbits of vertices if $p=2$, 4 if $p=3$ and in general $\frac{1}{24}(p-1)(5p+29)$.
\end{proof}

\begin{cor} \label{cor: rational homology of CPhi}
	Suppose $p$ is prime, $n=p+1$, and $\Phi$ is an order $p$ element of $\Out(F_n)$ that is not represented by an element of $\Aut(F_n)$ of order $p$. Then the rational homology of $C:= C\langle \Phi \rangle $ is given by 
	\begin{align*}
		\dim(H_0(C;\mathbb{Q})) &= 1 \\
		\dim(H_1(C;\mathbb{Q})) &= \begin{cases*}
			0, \qquad \qquad \qquad \qquad \text{ if } p=2,3 \\
			\frac{1}{24}(p-7)(p-5), \qquad \text{if } p \geq 5 
		\end{cases*} \\
		\dim(H_k(C;\mathbb{Q})) &= 0 \qquad \qquad \qquad \qquad \quad \text{ for } k \geq 2. \\
	\end{align*}
\end{cor}

\begin{proof}
	Since the action of Proposition~\ref{prop:action} is with finite stabilisers, a Mayer--Vietoris argument \cite[Chapter VII.9]{BrownCohomology} gives that the rational homology of $C\langle\Phi\rangle$ agrees with the (rational) homology of the quotient graph. The dimensions of the first homology groups then follow from the calculations of the number of edge and vertex orbits for this action.
\end{proof}

\begin{rem} \label{rem: transfer}
    Since $A$ is finite index in $\Out(F_m)$, there is an injective transfer map $\mathrm{tr}: H_\ast (\Out(F_m);\mathbb{Q}) \to H_\ast(A;\mathbb{Q})$ \cite[Chapter III.9]{BrownCohomology}, and any odd dimensional classes in the homology of $\Out(F_m)$ give rise to odd dimensional classes in the homology of $A$. Working rationally, this is isomorphic to the homology of $C\langle\Phi \rangle$, and so we will also see non-trivial classes for the centraliser arising this way. The earliest odd-dimensinal class that occurs this way is for $p=2$ and $n=13$, where we will see the transfer of Bartholdi's class in $H_{11}(\Out(F_7);\Q)$: the classes witnessed by Corollary~\ref{cor: rational homology of CPhi} occur in much lower dimensions and are not in the image of the transfer map.
\end{rem}

\section{The \texorpdfstring{$p$}{p}-adic Farrell--Tate \texorpdfstring{$K$}{K}-theory of \texorpdfstring{$\Out(F_{n})$}{Out(Fn)} for \texorpdfstring{$n=p-1,p,p+1,p+2,p+3$}{n=p-1,p,p+1,p+2,p+3}. } \label{Sec: Main calc}

In this section we start investigating the $p$-adic Farrell--Tate $K$-theory of $\Out(F_{n})$. Proposition \ref{prop:Tate K-theory} can be applied to $\Out(F_{n})$ since it admits a finite $\underline{E}G$ by \cite{CullerVogtmann1986CVn} and \cite{KrsticVogtmann1993}. We give a full computation of $\widehat{K_p}^*(B\Out(F_{n}))$ in the following cases: 
\begin{itemize}
    \item For $p \geq 3$ and $n=p-1,p$;
    \item For $p \geq 5$ and $n=p+1,p+2$;
    \item For $p \geq 7$ and $n=p+3$. 
\end{itemize}

It turns out that the cases $n=p-1$, $p$, $p+2$ follow quite easily from previously known results and Proposition \ref{prop:Tate K-theory}. The case $n=p+3$ additionally requires a recent calculation of Satoh \cite{Satoh24}. Finally, the case $n=p+1$ essentially uses the new results of Section \ref{section: Centraliser of Phi}. Note that at various points we require information about rational cohomology, and we have at our disposal results about rational homology: with untwisted coefficients these are linear duals, and hence will always have the same dimension.

\subsection{The case \texorpdfstring{$n=p-1$}{n=p-1}} \label{n is p-1} Let $p \geq 3$ (the case $p=2$ is straightforward). It follows from \cite{GMV98} and \cite[Proposition 3.1.2]{Chen97} that $\theta_{00}$ represents the only conjugacy class of order $p$ elements in $\Out(F_{p-1})$. By Proposition \ref{prop: rosethetacentralisers}
and Proposition \ref{prop:Tate K-theory}, since the centralisers are finite, we obtain
\[\widehat{K_p}^0(B\Out(F_{p-1})) \cong \Q_p,\]
and 
\[\widehat{K_p}^1(B\Out(F_{p-1}))=0.\]

\subsection{The case \texorpdfstring{$n=p$}{n=p}} \label{n is p} Here we assume $p \geq 3$. The case $p=2$ is considered separately below in Subsection \ref{n=2 case}. By \cite[Proposition 3.1.2]{Chen97} the group $\Out(F_p)$ has exactly $2$ conjugacy classes of order $p$ elements, given by $\theta_{01}$ and $R_p$. 
By Proposition \ref{prop: rosethetacentralisers} and Proposition \ref{prop:Tate K-theory}, we obtain
\[\widehat{K_p}^0(B\Out(F_{p})) \cong \Q_p \oplus \Q_p,\]
and 
\[\widehat{K_p}^1(B\Out(F_{p}))=0.\]

We conclude that the groups $\Out(F_{p})$ and $\Out(F_{p-1})$ satisfy weak duality in $p$-adic $K$-theory for $p \geq 3$. This is not the case for $\Out(F_{p+1})$ as we will see below. 

\subsection{The case \texorpdfstring{$n=p+1$}{n=p+1}} \label{n is p+1} This case is more involved than the previous cases. Let $p \geq 5$. 
It follows from \cite[Proposition 3.1.1]{Chen97} that there are exactly $4$ conjugacy classes of order $p$ elements in $\Aut(F_{p+1})$, denoted by $R_p$, $\theta_{02}$, $\theta_{20}$ and $\theta_{11}$.  The dual elements $\theta_{20}$ and $\theta_{02}$ become conjugate in $\Out(F_{p+1})$ \cite[Lemma 3.1.2]{Chen97}. By combining this with Corollary \ref{cor: strange order p element}, we conclude that $\Out(F_{p+1})$ has $4$ conjugacy classes of order $p$ elements represented by $R_p$, $\theta_{02}$, $\theta_{11}$ and $\Phi$. 
Next, using Proposition \ref{prop: rosethetacentralisers}, one has 
\[C\langle R_p \rangle \cong \Z/p \times ((\Z \rtimes \Z/2) \rtimes \Z/2),\]
where both copies of $\Z/2$ act by the sign. Further, we have 
\[C\langle \theta_{02} \rangle \cong \Z/p \times \Aut(F_2),\]
and
\[C\langle \theta_{11} \rangle \cong \Z/p \times ((\Z/2 \times \Z/2) \rtimes \Z/2),\]
where $\Z/2$ acts on $\Z/2 \times \Z/2$ by flipping the factors. The centraliser of $\Phi$ and its rational cohomology were studied in Section  \ref{section: Centraliser of Phi}. 

\begin{thm} \label{them: main computation} Let $p \geq 5$ be a prime. Then we have
\[\widehat{K_p}^0(B\Out(F_{p+1})) \cong \Q_p^4,\]
and 
\[\widehat{K_p}^1(B\Out(F_{p+1})) \cong \begin{cases} 0, & \text{if}  \;\; p=5,7\\ \Q_p^{\frac{1}{24}(p-7)(p-5)}, &  \text{if} \;\; p \geq 11  \end{cases}. \]

\end{thm}

\begin{proof} By Proposition \ref{prop:Tate K-theory}, the equivariant Chern character induces an isomorphism $(m=0,1)$:
\begin{align*}& \widehat{K_p}^m(BG) \cong  \prod_{i \in \Z} (H^{2i+m}(C   \langle R_p  \rangle; \Q_p) \times \\ &\times H^{2i+m}(C   \langle \theta_{02}  \rangle; \Q_p) \times  H^{2i+m}(C   \langle \theta_{11}  \rangle; \Q_p) \times  H^{2i+m}(C   \langle \Phi  \rangle; \Q_p)). \end{align*}

\noindent Since finite groups are rationally acyclic, it follows from \cite{HatcherVogtmann98} that the centraliser $C\langle \theta_{02} \rangle$ is rationally acyclic. The centraliser $C\langle \theta_{11} \rangle$ is finite and hence rationally acyclic. Finally, by twice using the Lyndon--Hochschild--Serre spectral sequence, 
we can see that $C\langle R_p \rangle$ is also rationally acyclic. Now the result follows from the calculation of $H^{*}(C   \langle \Phi  \rangle; \Q_p)$ which was done in Corollary \ref{cor: rational homology of CPhi}. \end{proof} 

We conclude that the groups $\Out(F_{p+1})$ for $p \geq 11$ do not satisfy weak duality in $p$-adic $K$-theory. The Farrell--Tate $K$-theory Euler characteristic tends to $-\infty$ as $p$ tends to $\infty$ showing that the failure of weak duality increases as the prime $p$ grows. The element $\Phi$ plays a crucial role in this, since the similar phenomenon does not occur for the cases $n=p-1$, $p$, $p+2$ and $p+3$. 

\subsection{The case \texorpdfstring{$n=p+2$}{n=p+2}} \label{n is p+2} Let $p \geq 5$. By \cite[Proposition 3.1.2]{Chen97}, the conjugacy classes of order $p$ elements are represented by $R_p$, $\theta_{03}$, and $\theta_{12}$. By Proposition \ref{prop: rosethetacentralisers}, the centralisers are given as follows: 
\[C\langle R_p \rangle \cong \Z/p \times ((F_2 \rtimes \Aut(F_2)) \rtimes \Z/2),\]
\[C\langle \theta_{03} \rangle \cong \Z/p \times \Aut(F_3),\]
and 
\[C\langle \theta_{12} \rangle \cong \Z/p \times \Z/2 \times \Aut(F_2).\]
It follows from the proof of \cite[Proposition 3.4.4, page 56]{Chen97} and the Lyndon--Hochschild--Serre spectral sequence that $C\langle R_p\rangle$ is rationally acyclic. The centralisers $C\langle \theta_{03} \rangle$ and $C\langle \theta_{12} \rangle$ are rationally acyclic as well by \cite{HatcherVogtmann98}. Hence using Proposition \ref{prop:Tate K-theory}, we obtain 
\[\widehat{K_p}^m(B\Out(F_{p+2})) \cong \begin{cases} \Q_p^3, \qquad \text{if}  \; m=0\\ 0, \qquad \;\;  \text{if} \;\; m=1  \end{cases}.\]
In particular, $\Out(F_{p+2})$ satisfies weak duality in $p$-adic $K$-theory for $p \geq 5$.

\subsection{The case \texorpdfstring{$n=p+3$}{n=p+3}} \label{n is p+3} Let $p \geq 7$. Again by \cite[Proposition 3.1.2]{Chen97}, the conjugacy classes of order $p$ elements are represented by $R_p$, $\theta_{04}$, $\theta_{13}$, and $\theta_{22}$. By Proposition \ref{prop: rosethetacentralisers}, the centralisers are given as follows:
\[C\langle R_p\rangle \cong \Z/p \times ((F_3 \rtimes \Aut(F_3)) \rtimes \Z/2),\]
\[C\langle \theta_{04} \rangle \cong \Z/p \times \Aut(F_4),\]
\[C\langle \theta_{13} \rangle \cong \Z/p \times \Z/2 \times \Aut(F_3),\]
and 
\[C\langle \theta_{22} \rangle \cong \Z/p \times (\Aut(F_2) \times \Aut(F_2)) \rtimes \Z/2.\]
By \cite{HatcherVogtmann98} and Lyndon--Hochschild--Serre spectral sequence arguments, the centralisers $C\langle \theta_{13} \rangle$ and $C\langle \theta_{22} \rangle$ are rationally acyclic. Again by \cite{HatcherVogtmann98} and Gerlits' thesis \cite[Theorem 3.3]{Gerlits}, we know that $H_i(\Aut(F_4); \mathbb{Q})=0$, unless $i=0,4$ and $H_4(\Aut(F_4); \mathbb{Q}) \cong \Q$, implying that $H^4(C\langle \theta_{04} \rangle; \Q) \cong \Q$ and $H^i(C\langle \theta_{04} \rangle; \Q)=0$ if $i \neq 0,4$. Some more work is required to compute the rational cohomology of  $C\langle R_p\rangle$. This is done in Lemma \ref{lem: roseF3} below which shows that $H^4(C\langle R_p\rangle; \Q) \cong \Q$ and $H^i(C\langle R_p\rangle; \Q)=0$ if $i \neq 0,4$.  All in all, we get:
\[\widehat{K_p}^m(B\Out(F_{p+3})) \cong \begin{cases} \Q_p^6, \qquad \text{if}  \; m=0\\ 0, \qquad \;\;  \text{if} \;\; m=1  \end{cases}.\]
In particular, $\Out(F_{p+3})$ satisfies weak duality in $p$-adic $K$-theory for $p \geq 7$.

\begin{lem} \label{lem: roseF3} Let $l \geq 2$. Then the rational cohomology of the centraliser $C\langle R_l\rangle$ of $R_l \in \Out(F_{l+3})$ is given by
\[H^i(C\langle R_l\rangle; \Q) \cong \begin{cases}  \Q, \qquad \text{if} \qquad i=0,4 \\ 0, \qquad   \;\text{if} \qquad   i \neq 0,4.  \end{cases}\]
\end{lem}

\begin{proof} By Proposition \ref{prop: rosethetacentralisers} and a Lyndon--Hochschild--Serre spectral sequence argument, it suffices to compute
\[H^i(F_3 \rtimes \Aut(F_3); \Q)^{\Z/2}\]
for any $i \in \Z$. Here we recall that $\Z/2$ acts on $F_3 \rtimes \Aut(F_3)$ by sending $(x,\sigma)$ to $(x^{-1}, \alpha_x \sigma)$, where $x \in F_3$, and $\sigma \in \Aut(F_3)$, and $\alpha_x(t)=xtx^{-1}$ (see e.g. \cite[page 32]{Chen97}). Following \cite[(3.5.23), page 61]{Chen97}, we consider the extension
\[
1 \to F_3 \rtimes F_3 \to F_3 \rtimes \Aut(F_3) \to \Out(F_3) \to 1,\]
where $F_3$ acts on $F_3$ via the conjugation. This short exact sequence is in fact a short exact sequence of groups with $\Z/2$-actions. The action on $\Out(F_3)$ is trivial, the action on the middle term was described above and the action on $F_3 \rtimes F_3$ is given by
\[(x,y) \mapsto (x^{-1}, xy).\]
The group $F_3 \rtimes F_3$ is isomorphic to $F_3 \times F_3$ via the shear map (sending $(x,y)$ to $(xy, y)$) and the $\Z/2$-action on $F_3 \rtimes F_3$ corresponds to the flip action on $F_3 \times F_3$. Now, the Lyndon--Hochschild--Serre spectral sequence for this extension is given by:
\[E_{2}^{ij}=H^i(\Out(F_3); H^j(F_3 \times F_3; \Q)) \Rightarrow H^{i+j}( F_3 \rtimes \Aut(F_3); \Q),\]
where $\Out(F_3)$ acts on $H^*(F_3 \times F_3; \Q)$ via the diagonal action of $\Aut(F_3)$ on $F_3 \times F_3$. 
Since the cohomological dimension of $F_3$ is equal to $1$, the K\"unneth theorem gives the following description of the $E_2$-term: 
\begin{align*}E_2^{i0} \cong H^i(\Out(F_3); \Q),\\ E_2^{i1} \cong H^i(\Out(F_3);H^1(F_3;\Q) \oplus H^1(F_3;\Q)), \\ E_2^{i2} \cong H^i(\Out(F_3); H^1(F_3; \Q) ^{\otimes 2}), \\ E_2^{ij}=0, \quad j \geq 3.  \end{align*}
We know that $\Out(F_3)$ is rationally acyclic using \cite{Brady93, Ohashi}. By \cite[Lemma 3.2.2]{Chen97}, we also know that $H^i(\Out(F_3); H^1(F_3; \Q))=0$ for any $i \in \Z$, and $H^i(\Out(F_3); H^1(F_3; \Q)^{\otimes 2})=0$ unless $i=2$. Furthermore, again by \cite[Lemma 3.2.2]{Chen97}, one has $H^2(\Out(F_3); H^1(F_3; \Q) ^{\otimes 2}) \cong \Q$. Hence the Lyndon--Hochshild--Serre spectral sequence collapses and implies that $H^i(F_3 \rtimes \Aut(F_3); \Q)=0$ unless $i=0,4$ and 
\[H^4(F_3 \rtimes \Aut(F_3); \Q) \cong H^2(\Out(F_3); H^1(F_3; \Q) ^{\otimes 2}) \cong \Q.\]
Furthermore, this isomorphism is $\Z/2$-equivariant, implying that the group of invariants  $H^4(F_3 \rtimes \Aut(F_3); \Q)^{\Z/2}$ is isomorphic to 
\[H^2(\Out(F_3); H^1(F_3; \Q) ^{\otimes 2})^{\Z/2}.\]
Here we take $\Z/2$-fixed points with respect to the action coming from the $\Z/2$-action on $H^1(F_3; \Q) ^{\otimes 2}$ which is given by the sign and flipping the tensor factors. The sign appears because of the K\"unneth theorem. Since we are working rationally, one has 
\[ H^i(\Out(F_3); H^1(F_3; \Q)^{\otimes 2})^{\Z/2} \cong H^i(\Out(F_3); (H^1(F_3; \Q)^{\otimes 2})^{\Z/2}),\]
and since the $\Z/2$-action involves the sign, the $\Out(F_3)$-module $(H^1(F_3; \Q)^{\otimes 2})^{\Z/2}$ is isomorphic to the second exterior power $\bigwedge^2 H^1(F_3; \Q)$. Hence it remains to show that 
\[H^2(\Out(F_3); {\bigwedge}^2 H^1(F_3; \Q)) \cong \Q.\]
Satoh in \cite[Theorem 3]{Satoh24} computes the cohomology with coefficients in the dual module $\bigwedge^2 H_1(F_3; \Q)$: 
\[H^i(\Out(F_3); {\bigwedge}^{2} H_1(F_3; \Q)) \cong \begin{cases}  \Q, \qquad \text{if} \qquad i=2 \\ 0, \qquad   \;\text{if} \qquad   i \neq 2.  \end{cases}.\]
To make this calculation, Satoh uses the (finite) Brown cochain complex (also referred to as the Bredon cochain complex) of the equivariant spine $K_3$ of $CV_3$ with coefficients in $\bigwedge^2 H_1(F_3; \Q)$. For any $\Out(F_3)$-equivariant $l$-cell $\sigma$ of $K_3$, we have the corresponding summand $(\bigwedge^2 H_1(F_3; \Q))^{G_{\sigma}}$ in the $\Q$-vector space of $l$-cochains, where $G_{\sigma}$ is the (finite) stabiliser of $\sigma$. By \cite[Theorem 3]{Satoh24}, the Euler characteristic of this chain complex is equal to $1$. The Brown cochain complex for $\bigwedge^2 H^1(F_3; \Q)$ is different to the one for $\bigwedge^2 H_1(F_3; \Q)$. However, the $\Q$-dimensions of the modules of $l$-cochains agree for any $l$ since rationally taking invariants with respect to finite groups commutes with forming dual vector spaces. Hence the Euler characteristics for $H^*(\Out(F_3); {\bigwedge}^{2} H_1(F_3; \Q))$ and $H^*(\Out(F_3); {\bigwedge}^{2} H^1(F_3; \Q))$ agree and we obtain 
\[\sum_{i \geq 0} (-1)^i \dim H^i(\Out(F_3); {\bigwedge}^2 H^1(F_3; \Q))=1,\]
showing that there exists $i$ such that $\dim H^i(\Out(F_3); {\bigwedge}^2 H^1(F_3; \Q)) \neq 0$. We already know that all these cohomology groups are trivial except possibly for $i=2$, implying that  
\[H^2(\Out(F_3); {\bigwedge}^2 H^1(F_3; \Q)) \cong \Q\]
which completes the proof. 
\end{proof}

\section{Low dimensional calculations and odd classes in \texorpdfstring{$K$}{K}-theory of \texorpdfstring{$\Out(F_n)$}{Out(Fn)}}

In this section we attempt to calculate the $p$-adic Farrell--Tate $K$-theory of $\Out(F_n)$ when $n \leq 12$. For appropriate primes we manage to fully calculate these $K$-theory groups. As a consequence we see that there is an explicit class in $K^1(B\Out(F_{12})) \otimes_{\Z} \mathbb{Q}$ which does not come from the rational cohomology of $\Out(F_{12})$. With the methods previously available we cannot find such classes in the groups $K^1(B\Out(F_{n})) \otimes_{\Z}\Q$ for $n \leq 11$ without using computer calculations.

Table~\ref{t:Farrell--Tate} shows the values of the $p$-adic Farrell--Tate $K$-theory of $\widehat{K_p}^*(B\Out(F_n))$ for $p \leq 11$ and $n \leq 12$ which we can compute combining our methods with previous computations. The empty spaces indicate groups which are not computed.

\begin{table}[h!]
    \centering
    \begin{tabular}{|c|c|c|c|c|c|c|c|c|c|c|}
        \hline 
        p
        & \multicolumn{2}{c|}{2} 
        & \multicolumn{2}{c|}{3} 
        & \multicolumn{2}{c|}{5} 
        & \multicolumn{2}{c|}{7} 
        & \multicolumn{2}{c|}{11} \\ 
        \hline
        \diagbox[width=2em]{n}{} 
        & even & odd 
        & even & odd 
        & even & odd 
        & even & odd 
        & even & odd \\ 
        \hline
        2 & $\Q_2^4$ & 0  & $\Q_3$  & 0  & 0  & 0  & 0  & 0  & 0  & 0  \\ \hline
        3 &  &  & $\Q_3^2$  & 0  & 0  & 0  & 0  & 0  & 0  & 0  \\ \hline
        4 &  &  &  &  & $\Q_5$  & 0  & 0  & 0  & 0  & 0  \\ \hline
        5 &  &  &  &  & $\Q_5^2$  & 0  & 0  & 0  & 0  & 0  \\ \hline
        6 &  &  &  &  & $\Q_5^4$   & 0  & $\Q_7$  & 0  & 0  & 0  \\ \hline
        7 &  &  &  &  & $\Q_5^3$  & 0  & $\Q_7^2$  & 0  & 0  & 0  \\ \hline
        8 &  &  &  &  & $\Q_5^7$  & 0  & $\Q_7^4$  & 0  & 0  & 0  \\ \hline
        9 &  &  &  &  &  &  & $\Q_7^3$  & 0  & 0  & 0  \\ \hline
        10 &  &  &  &  &  &  & $\Q_7^6$  & 0  & $\Q_{11}$  & 0  \\ \hline
        11 &  &  &  &  &  &  & $\Q_7^5 \oplus H^{\ev}$  & $\Q_7 \oplus H^{\odd}$  &$\Q_{11}^2$  &0  \\ \hline
        12 &  &  &  &  &  &  &  &  & $\Q_{11}^4$   & $\Q_{11}$  \\ \hline
    \end{tabular}
    \caption{The $p$-adic Farrell--Tate $K$-theory of $\Out(F_n)$}
    \label{t:Farrell--Tate}
\end{table}

\noindent Here we use the notations $H^{\ev}=\prod_{i > 0}H^{2i}(F_{4} \rtimes \Aut(F_{4}); \Q_7)^{\Z/2}$ and $H^{\odd}=\prod_{i \geq 0}H^{2i+1}(F_{4} \rtimes \Aut(F_{4}); \Q_7)^{\Z/2}$. These groups have not been computed. Most of the entries in this table can be explained using Section~\ref{Sec: Main calc} 
and in particular Theorem~\ref{them: main computation}. 
We also rely on the low dimensional computations of the rational cohomology of $\Out(F_n)$ and $\Aut(F_n)$ due to Brady, Vogtmann, Hatcher, Gerlits, Ohashi and Bartholdi \cite{Brady93, HatcherVogtmann98, Vogtmann02, Gerlits, Ohashi, Bartholdi}. These computations combined with Table~\ref{t:Farrell--Tate} and L\"uck's Theorem (Theorem \ref{thm:Wolfgang}), give calculations for the rational $p$-adic $K$-theory $K_p^*(B\Out(F_n)) \otimes_{\Z} \Q$ for $p \leq 7$ and $n \leq 7$, which are shown in Table~\ref{t:p-adic}.

\begin{table}[h!] 
    \centering
    \begin{tabular}{|c|c|c|c|c|c|c|c|c|}
        \hline 
        p
        & \multicolumn{2}{c|}{2} 
        & \multicolumn{2}{c|}{3} 
        & \multicolumn{2}{c|}{5} 
        & \multicolumn{2}{c|}{7} \\ 
        \hline
        \diagbox[width=2em]{n}{} 
        & even & odd 
        & even & odd 
        & even & odd 
        & even & odd \\ 
        \hline
        2 & $\Q_2^5$ & 0  & $\Q_3^2$  & 0  & $\Q_5$  & 0  & $\Q_7$  & 0    \\ \hline
        3 &  &  & $\Q_3^3$  & 0  & $\Q_5$  & 0  & $\Q_7$   & 0    \\ \hline
        4 &  &  &  &  & $\Q_5^3$  & 0  & $\Q_7^2$  & 0   \\ \hline
        5 &  &  &  &  & $\Q_5^3$  & 0  & $\Q_7$   & 0    \\ \hline
        6 &  &  &  &  & $\Q_5^6$   & 0  & $\Q_7^3$  & 0    \\ \hline
        7 &  &  &  &  & $\Q_5^5$  & $\Q_5$  & $\Q_7^4$  & $\Q_7$   \\ \hline
        
    \end{tabular}
    \caption{The rationalised $p$-adic $K$-theory of $\Out(F_n)$}
    \label{t:p-adic}
\end{table}

We give details of the most interesting and exceptional entries in these tables below. The rest follow easily.

\subsection{\texorpdfstring{$K$}{K}-theory of \texorpdfstring{$\Out(F_2)$}{Out(F2)}} \label{n=2 case} It follows from \cite{Nielsen24} that the natural map $\Out(F_2)$ to $\GL_2(\Z)$ is an isomorphism. On the other hand, $\GL_2(\Z)$ is isomorphic to the amalgamated product
\[D_4 \ast_{D_2} D_{6}.\]
Thus the $K$-theory of $B\Out(F_2)$ can be calculated using a Mayer--Vietoris sequence argument. 
The relevant primes here are $p=2,3$ as $\Out(F_2)$ only contains $2$ and $3$ primary torsion. 

For Farrell--Tate $K$-theory we get
\[\widehat{K_2}^m(B\Out(F_{2})) \cong \begin{cases} \Q_2^4, \qquad  \text{if} \;    m=0\\ 0, \qquad \;\; \text{if} \;  \; m=1 \end{cases}, \]
and 
\[\widehat{K_3}^m(B\Out(F_{2})) \cong \begin{cases} \Q_3, \qquad \text{if}  \; m=0\\ 0, \qquad \;\;  \text{if} \;\; m=1  \end{cases}.\]
The latter agrees with the calculation in Subsection \ref{n is p-1}. 
In fact in this case one can fully compute the (integral) $K$-theory groups, though we do not go into details here. 
We only summarise the rational $K$-theory groups
\[K^m (B\Out(F_{2})) \otimes_{\Z} \Q \cong \begin{cases} \Q \oplus \Q_2^4 \oplus \Q_3, \qquad \text{if}  \;\; m=0\\ 0, \qquad \;\;\;\;\;\;\;\;\;\;\;\;\;\;\;\;\;\; \text{if} \;\; m=1  \end{cases}.\]
In particular, $\widehat{K_p}^1(B\Out(F_{2}))$ vanishes for any prime $p$.

\subsection{\texorpdfstring{$K$}{K}-theory of \texorpdfstring{$\Out(F_7)$}{Out(F7)}} Assume that $p \geq 5$. The relevant primes in this case are $p=5,7$. We only consider the case $p=5$ since $p=7$ is analogous. 

Using Subsection \ref{n is p+2}, we obtain 
\[\widehat{K_5}^m(B\Out(F_{7})) \cong \begin{cases} \Q_5^3, \qquad \text{if}  \; m=0\\ 0, \qquad \;\;  \text{if} \;\; m=1  \end{cases}.\]
Bartholdi computed the rational homology of $\Out(F_7)$ \cite{Bartholdi}: it has two even and one odd dimensional classes. By Bartholdi's calculation and Theorem \ref{thm:Wolfgang}, we get a full calculation of the rational $5$-adic $K$-theory:
\[K_5^m(B\Out(F_{7})) \otimes_{\Z} \Q \cong \begin{cases} \Q_5^5, \qquad \text{if}  \; m=0\\ \Q_5, \qquad \;  \text{if} \;m=1  \end{cases}.\]
This is the first instance when an odd class in the $K$-theory of $\Out(F_n)$ appears, though the Farrell--Tate $K$-theory is still even. 


\subsection{The Farrell--Tate \texorpdfstring{$K$}{K}-theory of \texorpdfstring{$\Out(F_8)$}{Out(F8)}} \label{F8} The $p=7$ case follows immediately from Theorem \ref{them: main computation}. We have to carefully consider the $p=5$ case since $8>7=2 \cdot 5-3$. The group $\Out(F_8)$ contains, up to conjugacy, one rank $2$ elementary abelian $5$-subgroup (i.e. a subgroup isomorphic to $\Z/5 \times \Z/5$). For conjugacy classification of order $5$ elements, one needs to use the results of Glover and Henn \cite[Proposition 1.3] {GloverHenn10}. We have the rose and theta elements $R_5$, $\theta_{04}$, $\theta_{13}$ and $\theta_{22}$. Additionally there is one extra conjugacy class of order $5$ elements represented by $\Delta$, a diagonal embedding into 
a rank $2$ elementary abelian $5$-subgroup. The centraliser of $\Delta$ is rationally acyclic by \cite[Proposition 1.3] {GloverHenn10}. Further, by Chen's thesis (Proposition \ref{prop: rosethetacentralisers}), we have
\[C\langle R_5\rangle \cong \Z/5 \times ((F_3 \rtimes \Aut(F_3)) \rtimes \Z/2),\]
\[C\langle \theta_{04} \rangle \cong \Z/5 \times \Aut(F_4),\]
\[C\langle \theta_{13} \rangle \cong \Z/5 \times \Z/2 \times \Aut(F_3),\]
and 
\[C\langle \theta_{22} \rangle \cong \Z/5 \times (\Aut(F_2) \times \Aut(F_2)) \rtimes \Z/2.\] 
The rational cohomology of $C\langle R_5\rangle$ was computed in Lemma \ref{lem: roseF3}. Further, by \cite{HatcherVogtmann98} the centralisers $C\langle \theta_{13} \rangle$ and $C\langle \theta_{22} \rangle$ are rationally acyclic. Again by \cite{HatcherVogtmann98} and Gerlits' thesis \cite[Theorem 3.3]{Gerlits}, we know that $H_i(\Aut(F_4); \mathbb{Q})=0$, unless $i=0,4$ and $H_4(\Aut(F_4); \mathbb{Q}) \cong \Q$. Hence by Proposition \ref{prop:Tate K-theory}, we get
\[\widehat{K_5}^m(B\Out(F_{8})) \cong \begin{cases} \Q_5^7, \qquad \; \text{if}  \; m=0\\ 0, \qquad \;\;\;  \text{if} \;\; m=1  \end{cases}.\]
We cannot fully compute the rational $5$-adic $K$-theory since the rational cohomology of $\Out(F_8)$ has not been calculated.

\subsection{The Farrell--Tate \texorpdfstring{$K$}{K}-theory of \texorpdfstring{$\Out(F_{11})$}{Out(F11)}} Here the relevant primes are $p \leq 11$. This is the first instance when the Farrell--Tate $K$-theory has odd classes, though we cannot see these classes explicitly. 

The conjugacy classes of $7$-torsion elements are represented by $R_7$, $\theta_{05}$, $\theta_{14}$, $\theta_{23}$ and the centralisers are given as follows (Proposition \ref{prop: rosethetacentralisers}):
\[C \langle R_7 \rangle  \cong \Z/7 \times (F_{4} \rtimes \Aut(F_{4})) \rtimes \Z/2,\]
\[C \langle \theta_{05} \rangle \cong \Z/7 \times \Aut(F_5),\]
\[C \langle \theta_{14} \rangle \cong \Z/7 \times \Z/2 \times \Aut(F_4),\]
and
\[C \langle \theta_{23} \rangle \cong \Z/7 \times \Aut(F_2) \times \Aut(F_3).\]
Gerlits in \cite{Gerlits} showed using computer calculations that $H_7(\Aut(F_5); \Q)=\Q$ and $H_i(\Aut(F_5); \Q)=0$ unless $i=0,7$. Using \cite{Gerlits, HatcherVogtmann98} and Proposition \ref{prop:Tate K-theory}, we get
\[\widehat{K_7}^m(B\Out(F_{11})) \cong \begin{cases} \Q_7^5 \oplus H^{\ev}(F_{4} \rtimes \Aut(F_{4}); \Q_7)^{\Z/2} , \qquad \text{if}  \;\; m=0\\ \Q_7 \oplus H^{\odd}(F_{4} \rtimes \Aut(F_{4}); \Q_7)^{\Z/2}, \qquad  \text{if} \;m=1  \end{cases}.\]
Here $H^{\ev}$ and $H^{\odd}$ are direct sums of positive even and odd dimensional rational cohomologies respectively. We are not aware of a computation of the groups $H^{*}(F_{4} \rtimes \Aut(F_{4}); \Q)$ or $H^{*}(F_{4} \rtimes \Aut(F_{4}); \Q)^{\Z/2}$. The latter is closely related to $H^{*}(\Out(F_{4}) ; \bigwedge^2 H^1(F_4;\Q))$ and $H^{*}(\Out(F_{4}) ; H^1(F_4;\Q))$  which also seem to be unknown. We conclude with the observation that this is the first appearance of an odd dimensional class in $p$-adic Farell--Tate $K$-theory. However the construction of this class relies on the computer calculations of Gerlits \cite{Gerlits}. 

The $p=11$ case is straightforward. 


\subsection{The Farrell--Tate \texorpdfstring{$K$}{K}-theory of \texorpdfstring{$\Out(F_{12})$}{Out(F12)}} Again the relevant primes here are $p \leq 13$ and we are able to do full computations only in the cases $p=11, 13$. The case $p=13$ is straightforward. For $p=11$ we detect an explicit odd dimensional class in $K_{11}^1(B\Out(F_{12})) \otimes_{\Z} \Q$. Indeed, by Theorem \ref{them: main computation}, we obtain 
\[\widehat{K_{11}}^m(B\Out(F_{12})) \cong \begin{cases} \Q_{11}^4 , \qquad \text{if}  \; m=0\\ \Q_{11}, \qquad   \text{if} \;\;m=1\end{cases}.\]
More explicitly, the $\Q_{11}$-vector space $\widehat{K_{11}}^1(B\Out(F_{12}))$ maps isomorphically to 
\[H^1(C \langle \Phi \rangle ; \Q_{11})\]
using the equivariant Chern Character. This group is computed in Corollary \ref{cor: rational homology of CPhi} and is isomorphic to $\Q_{11}$. Hence one detects a class in $K^1(B\Out(F_{12})) \otimes_{\Z} \Q$ which maps to the generator of \[H^1(C \langle \Phi \rangle; \Q_{11})\] under the equivariant Chern character isomorphism. To our knowledge this is the first non-trivial class detected in the odd $K$-theory of $\Out(F_n)$ which does not require computer calculations. 

\begin{rem} One can construct non-trivial classes in $\widehat{K_{p}}^1(B\Out(F_{n}))$ for lower primes using Remark \ref{rem: transfer}. For instance, when $p=2$ and $n=13$, the transfer of Bartholdi's class in $H_{11}(\Out(F_7) ;\Q)$ produces a non-trvial class in $H_{11}(C\langle \Phi \rangle; \Q)$ which in turn gives a non-trivial class in $\widehat{K_{2}}^1(B\Out(F_{13}))$. However, clearly the construction of this class in Farrell--Tate $K$-theory needs ~Bartholdi's computer calculation.   

\end{rem}

\bibliographystyle{alpha}
\bibliography{biblio}
\end{document}